\documentclass[11pt,a4paper,reqno,oneside]{amsart}

\usepackage{amsrefs}
    \usepackage{latexsym}
    \usepackage{appendix}
    \usepackage{amsfonts}
    \usepackage{amsmath}
    \usepackage{amsthm}
    \usepackage{verbatim}
    \usepackage{amsfonts,amssymb}
\usepackage{mathrsfs}
\usepackage{yhmath}
\usepackage{amsmath}
\usepackage{amssymb}
\usepackage{amsfonts}
\usepackage[pdftex]{graphicx}
\usepackage{epstopdf}
\usepackage{tabularx}
\usepackage{extarrows}
  \usepackage{float}
    \usepackage{booktabs}
    \usepackage{multirow}
\usepackage{mathrsfs}
\usepackage{url}
\usepackage{multicol} 
\usepackage[linktocpage = true]{hyperref}
\hypersetup{
 colorlinks = true,
 citecolor = blue,
}
\usepackage{xcolor}
\usepackage{geometry}
\theoremstyle{plain}

\newtheorem{rem}{Remark}
\newtheorem{lem}{Lemma}
\newtheorem{thm}{Theorem}
\newtheorem{nota}{Notation}
\numberwithin{equation}{section}
\newcommand{\A}{\mathcal{A}}

\newcommand{\R}{\mathbb{R}}

\numberwithin{equation}{section}
\numberwithin{equation}{section}
\numberwithin{equation}{section}
\numberwithin{equation}{section}
\numberwithin{equation}{section}
\numberwithin{equation}{section}
\numberwithin{equation}{section}
\numberwithin{equation}{section}
\numberwithin{equation}{section}
\numberwithin{equation}{section}
\numberwithin{equation}{section}
\numberwithin{equation}{section} 

\def\ep{\varepsilon}

\allowdisplaybreaks

\begin{document} 
\title[Error bounds for the asymptotic expansions of the Jacobi polynomials]{Error bounds for the asymptotic expansions of the Jacobi polynomials}
\thanks{ Research partially supported by  Guangdong Basic and Applied Basic Research Foundation No. 2021A1515110654(Huang).}

\author{Xiao-Min Huang}
\address{School of Mathematics and statistics, 
Guangdong University of Technology, China}
\email{\href{mahuangxm@gdut.edu.cn}{mahuangxm@gdut.edu.cn}}

\author{Yu Lin$^*$}
\address{Department of Mathematics,
South China University of Technology, China}
\email{\href{scyulin@scut.edu.cn}{scyulin@scut.edu.cn}}

\author{Xiang-Sheng Wang}
\address{Department of Mathematics, University of Louisiana at Lafayette, Lafayette, LA 70503, USA}
\email{\href{xswang@louisiana.edu}{xswang@louisiana.edu}}

\author{R. Wong}
\address{Liu Bie Ju Centre for Mathematical Sciences,
City University of Hong Kong, Hong Kong,
China}
\email{\href{rodscwong@gmail.com}{rodscwong@gmail.com}}

\allowdisplaybreaks

\begin{abstract}
 This paper aims to derive explicit and computable error bounds for the asymptotic expansion of the Jacobi polynomials as their degree approaches infinity, using an integral method. The analysis focuses on the outer or oscillatory region of these polynomials. A novel technique is introduced to address the challenges posed by the logarithmic singularity in the phase function of the integral representation of Jacobi polynomials. A recurrence formula is also developed to compute the coefficients in the asymptotic expansions.

\end{abstract}

\maketitle{}



\section{Introduction}
The study of the asymptotic expansions of Jacobi polynomials as the polynomial degree approaches infinity has garnered considerable attention in the literature \cite{CI91,KMVV04,KV99,Szego79,WZ96,WZ03}. Various standard asymptotic techniques have been used to investigate these expansions, including the steepest descent method for integrals \cite{Wong89}, the WKB method tailored for differential equations \cite{Olver74}, the Deift–Zhou strategy for Riemann–Hilbert problems \cite{DKMVZ99,DZ93,KMVV04}, the asymptotic theory of difference equations \cite{KV99,Wong14}, and Darboux's approach \cite{WZ05}. Notably, Wong and Zhang, in \cite{WZ96}, provided an estimate for the error term in the asymptotic expansions of Jacobi polynomials that is of the same order as the first neglected term. In \cite{WZ03}, Wong and Zhao introduced an alternative approach to address a deficiency in the behavior of these expansions near the origin.   However, while these works provide valuable theoretical insights, fall short of offering explicit and computable error bounds.  
In \cite{Olver80}, Olver demonstrated that well-constructed error bounds for asymptotic
approximations can offer invaluable insight into the nature and reliability of the approximations,
enable somewhat unsatisfactory concepts and lead to significant extensions of asymptotic results. He urged asymptotic analysts to peer intently at the error term with the following reasons: (i) To legitimize computation.
(ii) To provide analytical insight.
(iii) To avoid unsatisfactory concepts.
(iv) To facilitate extensions. 
As pointed out in \cite{Wong80}, it is challenging to calculate the upper limits of error terms obtained through conventional asymptotic techniques.
 To address this issue, we require fresh insights and more robust tools. By exploring these new tools, we will enhance their practicality for certain applications.

In this paper, we endeavor to derive explicit and feasible error bounds for the asymptotic expansions of the Jacobi polynomials. The primary hurdle lies in identifying an appropriate contour for the integral representation of the error term that allows for a computable error bound. The ‘adjacent saddles’ method, introduced by Berry and Howls \cite{BH91} and further refined in \cite{BHNOD18,B93}, suggests expressing the error term as an integral on steepest descent contours of neighboring saddle points. However, this technique is not directly applicable to Jacobi polynomials and numerous other orthogonal polynomials because of the presence of a branch point in the integrand. 
To overcome these obstacles, the authors of \cite{SNWW23} propose a ‘branch cut’ technique to study the Hermite polynomials, which involves modifying the contour of integration for the error term to align with the branch cut of the phase function. When studying the outer interval of the Jacobi polynomials, we encounter a challenge in the deformation of the contour due to the fact that both saddle points of the phase function lie on the steepest descent contour. As a result, the direct application of Cauchy's integral formula, as demonstrated in \cite{SNWW23}, is not straightforward. To overcome this obstacle, we need to divide the steepest descent contour, since Cauchy's integral formula typically requires the integrand to be analytic within a connected domain that contains the contour of integration and excludes any singularities within that domain.
However, the discussion for the Jacobi polynomials in the oscillatory interval becomes more complex than in the outer interval. In this case, the circumference of the steepest descent contour is difficult to describe, making it impractical to divide the contour as we did before. Therefore, we need to adopt a different approach to analyze this interval effectively.

The approaches in our paper not only circumvent the limitations posed by the `branch cut' technique but also provide a practical framework for computing error bounds in the context of Jacobi polynomials.

The integral representation of the Jacobi polynomials is given by 
\begin{equation}\label{Eq:Jacobi1}
P_n^{(\alpha,\beta)}(x)=\frac{(-1)^n(1-x)^{-\alpha}(1+x)^{-\beta}}{2^{n+1}\pi i}\int_{\mathcal L} (1-z)^{n+\alpha}(1+z)^{n+\beta}(z-x)^{-n-1}dz,
\end{equation}
 with $\alpha>-1$, $\beta>-1$,  $x\neq \pm1$, where $\mathcal L$ is extended in the positive sense along a closed curve around the point $x$, such that the points $\pm1$ lie neither on it nor in its interior (see \cite[p.72]{Szego79}).
Clearly, equation \eqref{Eq:Jacobi1} can be rewritten as 
\begin{align*}
P_n^{(\alpha,\beta)}(-x)=\frac{(1+x)^{-\alpha}(1-x)^{-\beta}}{2^{n+1}\pi i}\int_{\mathcal L} (1-z)^{n+\beta}(1+z)^{n+\alpha}(z-x)^{-n-1}dz=(-1)^n P_n^{(\beta,\alpha)}(x).
\end{align*}
By this relation, we may assume that $x\geqslant0$ and $x\neq1$.
It follows from the recurrence relation \cite[Equ. (4.5.4)]{Szego79}
\[P_{n}^{(\alpha+1,\beta)}(x) = \frac{2}{2n + \alpha + \beta + 2} \cdot \frac{(n + \alpha + 1)P_{n}^{(\alpha,\beta)}(x) - (n + 1)P_{n+1}^{(\alpha,\beta)}(x)}{1 - x},\]
and 
\[P_{n}^{(\alpha,\beta+1)}(x) = \frac{2}{2n + \alpha + \beta + 2} \cdot \frac{(n + \beta + 1)P_{n}^{(\alpha,\beta)}(x) + (n + 1)P_{n+1}^{(\alpha,\beta)}(x)}{1 + x},\]
we may assume without loss of generality that $-1<\alpha\leqslant 0$ and $\beta\leqslant \alpha$.

For $x>1$, we can denote the phase function in integral  \eqref{Eq:Jacobi1} by  
\begin{equation}\label{Eq:fz1}
f(z)=\log(z-x)-\log(z-1)-\log(z+1),
\end{equation}
which is analysis in $\mathbb C\setminus(-\infty, x]$.
The saddle points are defined as the zeros of $f'(z)$, which can be calculated to be 
\(
z_{\pm}=x\pm \sqrt{x^2-1}.
\)
For $x\in[0,1)$, the phase function in integral  \eqref{Eq:Jacobi1} shall be
\begin{equation}\label{Eq:fz2}
f(z)=\log(z-x)-\log(1-z)-\log(z+1),
\end{equation}
which is analytic in $\mathbb C\setminus(-\infty, x]\cup [1, +\infty)$.
Setting $f'(z)=0$, we find the critical points
\(
z_{\pm}=x\pm i\sqrt{1-x^2}.
\) 
For this reason, we will consider the following two cases. 

Case 1: the outer interval with $x=\cosh\gamma>1$ and $\gamma>0$ ; 

Case 2: the oscillatory interval with $x=\cos\gamma\in[0,1)$ and $\gamma\in(0,\frac{\pi}{2}]$.

The main result associates to the asymptotic expansion with error bound of the Jacobi polynomials for Case 1 can be derived as follows. 
\begin{thm}\label{Th:C1}
Let $-1<\beta\leqslant \alpha\leqslant 0$, $\gamma>0$ and $n$ be a positive integer, $\delta$ be given by
\begin{equation}\label{Eq:delta}
    \delta=\frac12\min\Big\{1, \sqrt{2\gamma},\sqrt{\log[4(\cosh\gamma+1)]-\gamma}, \sqrt{\log\frac{(\cosh\gamma+2)^2}{\cosh\gamma(\cosh\gamma-1)}}\Big\}.
\end{equation}
 Define that 
\[
c_p:=\left(1+\frac{1}{4}\left( p(p-1)\delta +2p+2\right)\left(\frac{p+1}{\delta^2 e}\right)^{p+1}\right)\frac{e^{2\gamma}}{4^p\delta^{2p+1}}\cdot\frac{ 10(\cosh\gamma+2)^5}{(\beta+1)(\cosh\gamma-1)},
\]
for any positive integer $p$. Then the Jacobi polynomials have the asymptotic expansion
\begin{equation}\label{Eq:Jacobiu1}
\begin{split}
P_n^{(\alpha,\beta)}(\cosh\gamma)
=&\frac{2^{\alpha+\beta}\cdot e^{\gamma (n+\alpha+\beta+1)}}{ (e^{\gamma}-1)^{\alpha}(e^{\gamma}+1)^{\beta}\sqrt{n\pi(e^{2\gamma}-1)}}\left(\sum_{j=0}^{p-1}\frac{\A_j(e^{\gamma})}{n^j}+\zeta_p(n,\alpha,\beta)\right),
\end{split}
\end{equation}
where the error term is bounded by
\begin{align}
|\zeta_p(n,\alpha,\beta)|\leqslant\frac{\hat c_p}{n^p},
\end{align}
with
\[
\hat c_p=|\A_p(e^{\gamma})|+\left|\frac{c_{p+1}\Gamma(p+\frac32)}{n}\right|
\]
the coefficients $\A_j(e^{\gamma})$ of this asymptotic expansion introduce in equation \eqref{Eq:Aj1}. 
The constant $\hat c_p$  satisfies $\hat c_p\to \A_p(e^{\gamma})$ as $n\to\infty$, which means that the error bound asymptotically matches the magnitude of the first neglected term in the expansion. The rational functions
$\A_p(e^{\gamma})=\frac{S_p(e^\gamma)}{T_p(e^\gamma)}$ (with $S_p(e^\gamma)$ and $T_p(e^\gamma)$ being degree-equivalent polynomials).
These rational functions $\A_p(e^{\gamma})$ can be computed via the recursive algorithm in Section \ref{APP:A}.
\end{thm}

The main result derives the asymptotic expansion of the Jacobi polynomials for Case 2, along with its corresponding error bound, as follows.

\begin{thm}\label{Th:C2}
Let $-1<\alpha\leqslant 0$, $\beta\leqslant \alpha$ and $0<\gamma<\frac{\pi}2$, $n $ be a positive integer. Define 
\[N=\frac{\alpha+\beta+1}{2}+n,\qquad \kappa=-\Big(\frac{\alpha}{2}+\frac{1}{4}\Big)\pi\]
and
\begin{align*}
\notag  
c_p^+= &\pi +24 \pi \gamma^{2\alpha}+  \frac{128\sqrt{3\pi}(4e+1)^{\alpha+1}}{(\beta+1)\gamma^{p+3}}+\frac{128\sqrt{3\pi}}{\gamma^{p+3}}\mathbf{B}(\alpha+1,\beta+1) + \frac{9280 \gamma^{\alpha} + 1184\pi(\pi-\gamma)\gamma^{\beta}}{(\beta+1)^{\frac{3}{2}}\gamma^{p+1}}.
\end{align*}
For any positive integer $p$, we have the asymptotic expansion
\begin{align}\label{Eq:PnC2}
\notag P_n^{(\alpha,\beta)}(\cos\gamma)
=&\frac{1}{\sin^{\alpha+\frac{1}{2}}\frac{\gamma}{2} \cos^{\beta+\frac{1}{2}}\frac{\gamma}{2}\sqrt {n\pi}}\Big\{\cos(N\gamma+\kappa ) \sum_{j=0}^{p-1}\frac{\Re\big( \A_j(e^{i\gamma})\big)}{n^{j}}\\
&+\sin(N\gamma+\kappa)\sum_{j=0}^{p-1}\frac{\Im\big( \A_j(e^{i\gamma})\big)}{n^{j}}+\varepsilon_p(n,\alpha,\beta)\Big\},
\end{align}
where
\begin{align*}
|\varepsilon_p(n,\alpha,\beta)|\leqslant\frac{\hat c_p}{n^p},
\end{align*}
with
\[
\hat c_p=\left|\Re \big(\A_p(e^{i\gamma})e^{(N\gamma+\kappa)i}\big)\right|+\frac{c_{p+1}^+\Gamma(p+\frac32)}{n}
\] and the coefficients $\A_j(e^{i\gamma})$ define in equation \eqref{Eq:Aj2}. 
Furthermore, as $n\to \infty$  the constant $\hat c_p$  converges to $\left|\Re \big(\A_p(e^{i\gamma})e^{(N\gamma+\kappa)i}\big)\right|$. This implies that for sufficiently large $n$, the error bound is approximately equal to the absolute value of the first term neglected in the asymptotic expansion. The coefficients
$\A_p(e^{i\gamma})$ are defined as rational functions of  $e^{i\gamma}$,  specifically given by $\A_p(e^{i\gamma})=\frac{S_p(e^{i\gamma})}{T_p(e^{i\gamma})}$, where $S_p(e^{i\gamma})$ and $T_p(e^{i\gamma})$ are polynomials of the same degree. 
These rational functions $\A_p(e^{i\gamma})$ can be computed using a recursive formula, which is detailed in the Section \ref{APP:A}.
\end{thm}

Note that $\A_0(e^{i\gamma})=\A_0(e^{\gamma})=1$ (refer to Section \ref{Sec:C1} and Section \ref{Sec:C2}).  A direct computation shows that the leading term of the asymptotic expansion in \eqref{Eq:Jacobiu1} is consistent with  \cite[ Equ. (8.21.9)]{Szego79}). Similarly, it is evident that  the leading term in \eqref{Eq:PnC2} agrees with \cite[ Equ. (8.21.10)]{Szego79}.

Throughout this paper, we investigate analytic functions with branch cuts along a real interval, say, $(c,d)$, which may be finite or infinite. While such a function $f$ may not be well-defined on the interval itself, it admits one-sided limits from the upper and lower half-planes. Specifically, for each $x\in(c,d)$, the limits  
\[
\lim_{\varepsilon \to 0^+} f(x \pm i \varepsilon)  
\]  
exist.  

For convenience, we introduce the notation $(c,d)^\pm$ to denote the limiting values of $(c,d) \pm i\varepsilon$ as $\varepsilon\to0^+$. This notation is meaningful only when $f$ is analytic in a neighborhood of $x\pm i\varepsilon$ for all $x\in(c,d)$ and $\varepsilon>0$. If $c$ or $d$ is finite, analogous definitions apply for half-open or closed intervals, such as $[c,d)^\pm$, $(c,d]^\pm$, or $[c,d]^\pm$.  

As an example, consider the square root function $f(z)=(z-1)^{1/2}$ with a branch cut along $(-\infty,1]$. For $x\in(-\infty,1)^\pm$, we have  
\[
f(x) = \lim_{\varepsilon \to 0^+} f(x \pm i \varepsilon) = \pm i \sqrt{1 - x}.  
\]  

This paper is organized as follows. In Section~\ref{Sec:C1}, we derive explicit and computable bounds for the error term of the Jacobi polynomials in the outer region. In Section~\ref{Sec:C2}, we derive analogous bounds for the oscillatory region. The recurrence formula for the coefficients in the asymptotic expansions of the Jacobi polynomials is presented in Section \ref{APP:A}. The proof of part of the lemmas will be given in Appendix \ref{APP:L}.

\section{Case 1: \(x>1\)}\label{Sec:C1}
Assume $x=\cosh \gamma>1$ with $\gamma>0$.
We rewrite equation \eqref{Eq:Jacobi1} as  
\begin{equation}\label{Eq:Jacobi2}
P_n^{(\alpha,\beta)}(x)=\frac{(x-1)^{-\alpha}(x+1)^{-\beta}}{2^{n+1}\pi i}\int_{\mathcal L} e^{-nf(z)}g(z)dz,
\end{equation}
where $f(z)$ is the phase function in equation \eqref{Eq:fz1} and 
\begin{equation}\label{Eq:gz1}
g(z)=\frac{(z-1)^{\alpha}(z+1)^{\beta}}{z-x}.
\end{equation}
Recalling that the saddle points
\(
z_{\pm}=x\pm \sqrt{x^2-1}=e^{\pm \gamma}
\), by choosing an appropriate cut for the logarithmic function, we have
\(
f(z_\pm)=\mp\gamma-\log 2.
\)
A simple calculation yields
\[
f''(z)=-\frac{1}{(z-x)^2}+\frac{2(z^2+1)}{(z^2-1)^2}.
\]
In particular, we have 
\[
f''(z_+)=\frac{2}{1-e^{2\gamma}}<0.
\]
In order to describe the steepest decent contour, we introduce the analytic function 
\begin{equation}\label{Eq:wz1}
\begin{split}
w(z)&=\left(f(z)-f(z_+)\right)^{1/2}=\left(-\log\frac{z^2-1}{z-x}+\log 2 +\gamma \right)^{1/2}, 
\end{split}
\end{equation}
for $z\in \mathbb C\setminus(-\infty, x]$.
As $z\to z_+$, we have the Taylor expansion:
\begin{align*}
f(z)=f(z_+)+\frac{f''(z_+)}{2}(z-z_+)^2+O\big((z-z_+)^3\big).
\end{align*}
We choose the square root function of equation \eqref{Eq:wz1} as 
\begin{align*}
w(z)=e^{-\frac{\pi }2i}\sqrt{\frac{-f''(z_+)}{2}}(z-z_+)+O\big((z-z_+)^2\big).
\end{align*}
Note that on the steepest decent/ascent curve, we have ${\rm Im}\big(f(z)-f(z_+)\big)=0$, i.e., ${\rm Im}f(z)=0$. A simple calculation gives 
 \begin{equation}\label{Eq:imft=0}
 \begin{cases}
 {\rm Im}\big(z+x+\cfrac{x^2-1}{z-x}\big)=0;\\
{\rm Re}\big(z+x+\cfrac{x^2-1}{z-x}\big)>0.
 \end{cases}
 \end{equation}
Set $z-x=a+ib$. We obtain from  \eqref{Eq:imft=0} two curves in the $z$ complex plane and $f(z)\geqslant f(z_+)$;
the steepest descent curve is defined by
\begin{align*}
\{z=x+a+ib \, |\, a^2+b^2=x^2-1 \}\cup \{z=x+a \, |\, -1< a+x<1\}.
\end{align*}
From  \eqref{Eq:imft=0}  and $f(z)\leqslant f(z_+)$; the steepest ascent curve is defined by
\begin{equation*}
\{z=x+a \, |\, a>0\}.
\end{equation*}

Now, we define the steepest decent curve $\Gamma$ as the union of the circle given in the above equation and two line segments connecting $1\pm i\ep$ to $z_-\pm i\ep$ with $\ep\to0^+$; see Fig. \ref{Fig:sd1}. 
\begin{figure}[H]
\center
\includegraphics[scale=0.6] {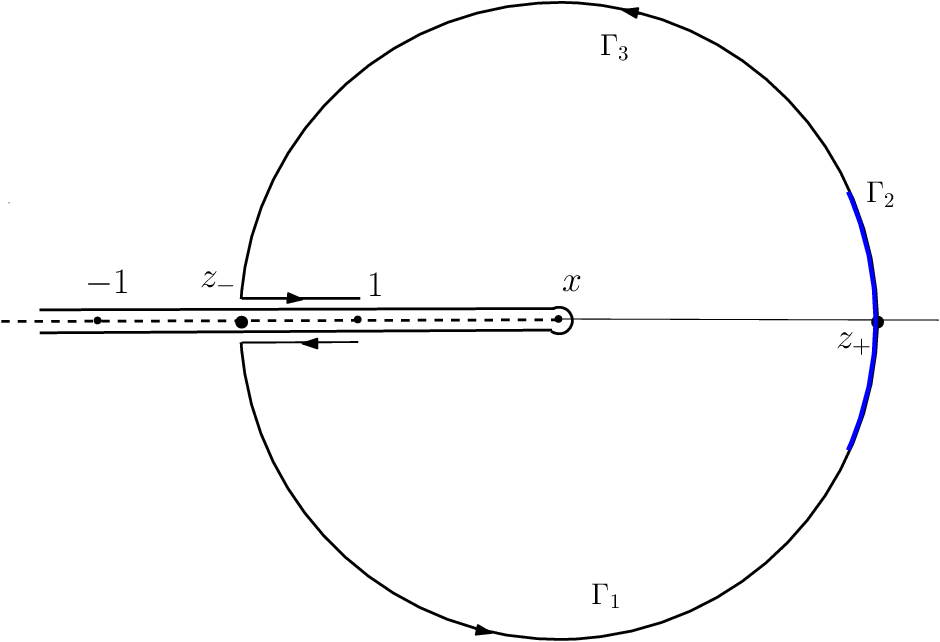}
 \caption{The steepest decent contour $\Gamma=\Gamma_1+\Gamma_2+\Gamma_3$.}   
 \label{Fig:sd1}  
\end{figure}

As in \cite{SNWW23}, we observe from equation \eqref{Eq:wz1} that when $z$ traverses along the steepest descent curve $\Gamma$, starting from $1$ in the lower plane and ending at $1$ in the upper plane, the behavior of the functions $f(z)$ and $w(z)$ can be described as follows. The function $f(z)$ initially decays from $+\infty$ to $f(z_+)$ (at $z = z_+$) and subsequently increases from $0$ to $+\infty$. Meanwhile, the function $w(z)$ exhibits a monotonic increase from $-\infty$ to $+\infty$.  

\begin{figure}[H]
\centering
\includegraphics[scale=0.5] {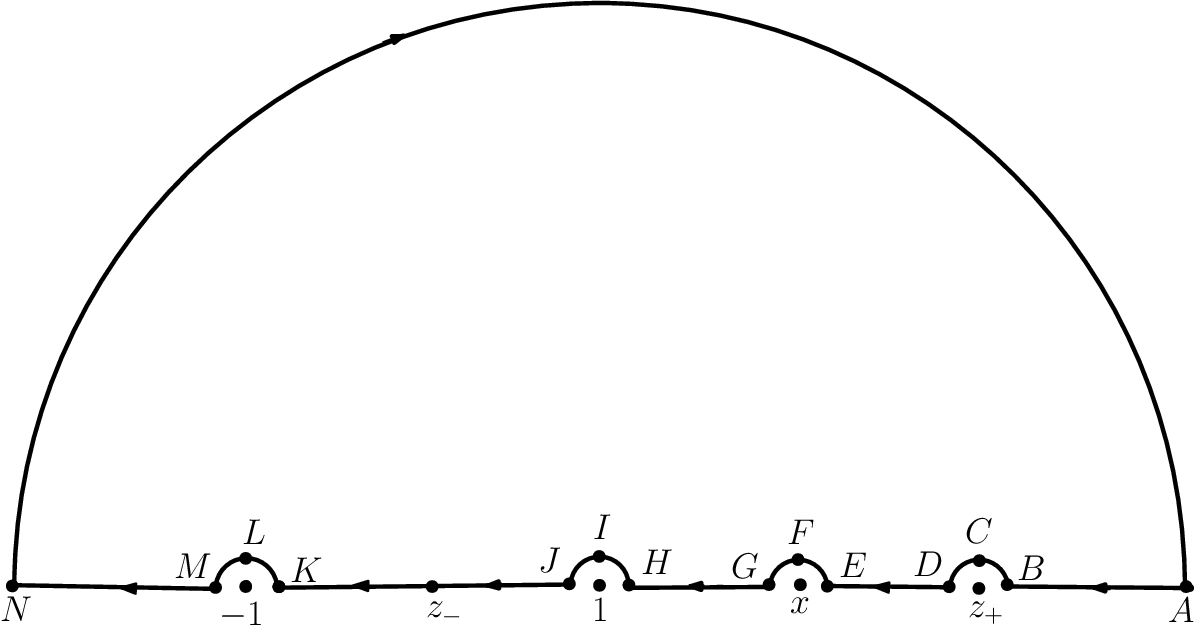}
\label{Fig:zcon1}
 \caption{ The $z-$plane.}   
\end{figure}
\begin{figure}[H]
\centering
\includegraphics[scale=0.6] {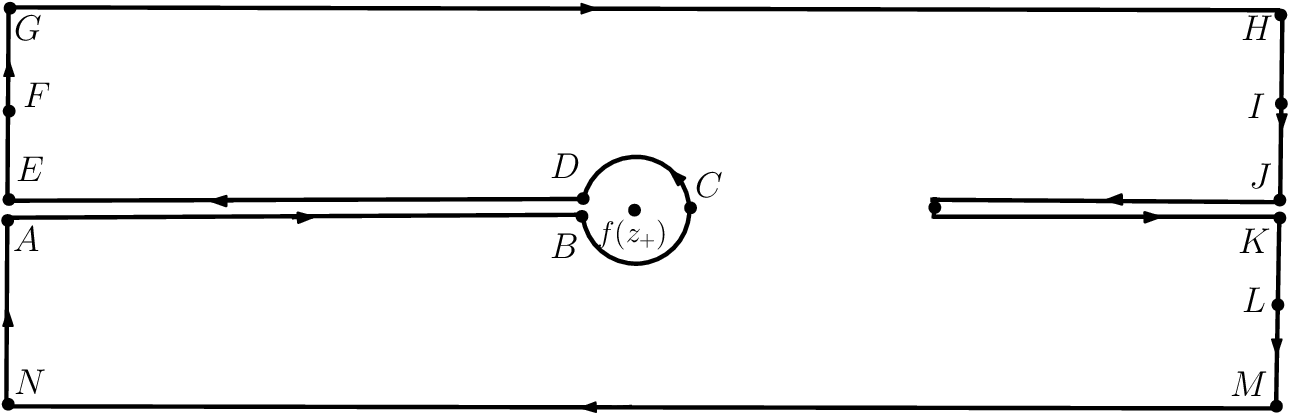}
 \caption{ The $f-$plane.}  
 \label{Fig:f1}   
\end{figure}
\begin{figure}[H]
\centering
\includegraphics[height=8cm,width=8cm] {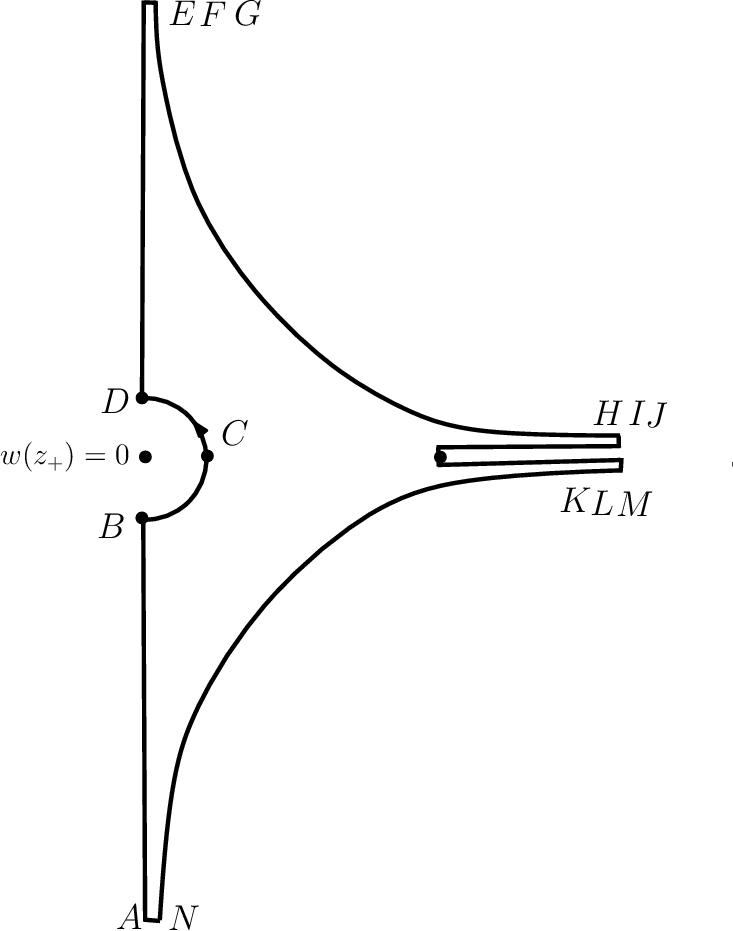}
 \caption{ The $w-$plane.}     
 \label{Fig:w1}
\end{figure}

To visualize the conformal maps $f(z)$ and $w(z)$, we plot a closed curve $ABCDEFGHIJKLMNA$ in the upper half of the $z$-plane, as shown in Fig. \ref{Fig:zcon1}, along with its corresponding images under the maps $f(z)$ and $w(z)$, shown in Fig. \ref{Fig:f1} and Fig. \ref{Fig:w1}, respectively. By symmetry, the images of a similar closed curve in the lower half of the $z$-plane can be obtained through a simple reflection.

 Let $\delta\in(0,1)$ be a constant as in equation \eqref{Eq:delta}. Recall that $x=\cosh \gamma$.
The parameter $\eta\in(0,\pi)$ is determined by the following equation
\begin{equation}\label{eta-delta}
    w(\cosh\gamma+\sinh\gamma e^{i\eta})=\delta.
\end{equation} 
Since both  saddle points  $z_{\pm}=e^{\pm \gamma}$ lie on the steepest decent contour of $f(z)$,  we adopt a strategic approach by dissecting this contour $\Gamma$ into distinct segments: $\Gamma_1+\Gamma_2+\Gamma_3$, as illustrated in Fig. \ref{Fig:sd1}, where
\begin{equation}\label{Eq:Gam1}
    \Gamma_1=[e^{-\gamma},1)^-\cup
\Gamma_{11},
\end{equation}
with
\begin{equation}\label{Eq:Gam11}
\Gamma_{11}=\{z|z=\cosh\gamma+\sinh\gamma e^{i\theta}, ~\theta\in (-\pi,-\eta)\};
\end{equation}
and
\[
\Gamma_2=\{z|z=\cosh\gamma+\sinh\gamma e^{i\theta}, ~\theta\in (-\eta,\eta)\};
\]
\[
\Gamma_3=\{z|z=\cosh\gamma+\sinh\gamma e^{i\theta}, ~\theta\in (\eta,\pi)\}\cup [z_-,1)^+.\]

By changing the contour $\mathcal L$ to the steepest decent $\Gamma=\Gamma_1+\Gamma_2+\Gamma_3$,  we can rewrite equation \eqref{Eq:Jacobi2} as 
\begin{equation}\label{Eq:Jc1}
\begin{split}
P_n^{(\alpha,\beta)}(\cosh\gamma)
=(\cosh\gamma-1)^{-\alpha}(\cosh\gamma+1)^{-\beta}\left(I_1+I_2+I_3\right),
\end{split}
\end{equation}
where 
\begin{equation}\label{Eq:Ij1}
I_j=\frac{e^{-nf(z_+)}}{2^{n+1}\pi i}\int_{\Gamma_j}\,e^{-n[f(z)-f(z_+)]}g(z)\,dz,  ~~~~j=1,2,3.
\end{equation}
As we shall see later, $I_1$ and $I_3$ are exponentially small with respect to $I_2$.

\subsection{Estimation of \( |I_2| \) }
By making  a change of variable $u=w(z)$, we can rewrite $I_2$ as 
\begin{equation}\label{Eq:J1I2}
I_2=
\frac{e^{-nf(z_+)}}{2^{n+1}\pi i}\int_{-\delta}^{\delta}\,e^{-nu^2}\frac{g(w^{-1}(u))}{w'(w^{-1}(u))}\,du.
\end{equation}
The Cauchy integral formula yields that
\begin{align}\label{Eq:ginu}
\frac{g(w^{-1}(u))}{w'(w^{-1}(u))}=\frac1{2\pi i}\int_{\Gamma^0}\frac{g(z)}{w(z)-u}\,dz
\end{align}
for $u\in [-\delta,\delta]$, and 
the contour  $\Gamma^0$  is a closed, clockwise-oriented, contour with  $\Gamma_2$ in its interior (see Fig. \ref{Fig:sdcga0}). More specifically, $\Gamma^0$ is composed of three distinct components
\[
\Gamma^0=\Gamma_x^0+\Gamma_{(x,o)^{\pm}}+\Gamma_o^0,
\]
which combines the arcs $\Gamma_x^0$ and $\Gamma_o^0$ and the horizontal segment $\Gamma_{(x,o)^{\pm}}$, defined by
\begin{align*}
\Gamma_x^0&=\bigg \{z|z=\cosh\gamma+\frac{\cosh\gamma-1}{4(\cosh\gamma+2)^4}e^{i\theta}, ~\theta \in \big(-\pi,\pi\big)\bigg\};\\
\Gamma_{(x,o)^{\pm}}&=\bigg \{z|z=\mu, ~\mu \in [-10(\cosh\gamma+1)+\cosh\gamma, ~\cosh\gamma-\frac{\cosh\gamma-1}{4(\cosh\gamma+2)^4}]^\pm \bigg\};
\\
\Gamma_o^0&=\bigg \{z|z=\cosh\gamma+re^{i\theta}, ~r=10(\cosh\gamma+1),~\theta \in \Big(-\pi,\pi)\bigg\}.
\end{align*} 

\begin{figure}[H]
\center
\includegraphics[scale=0.45]{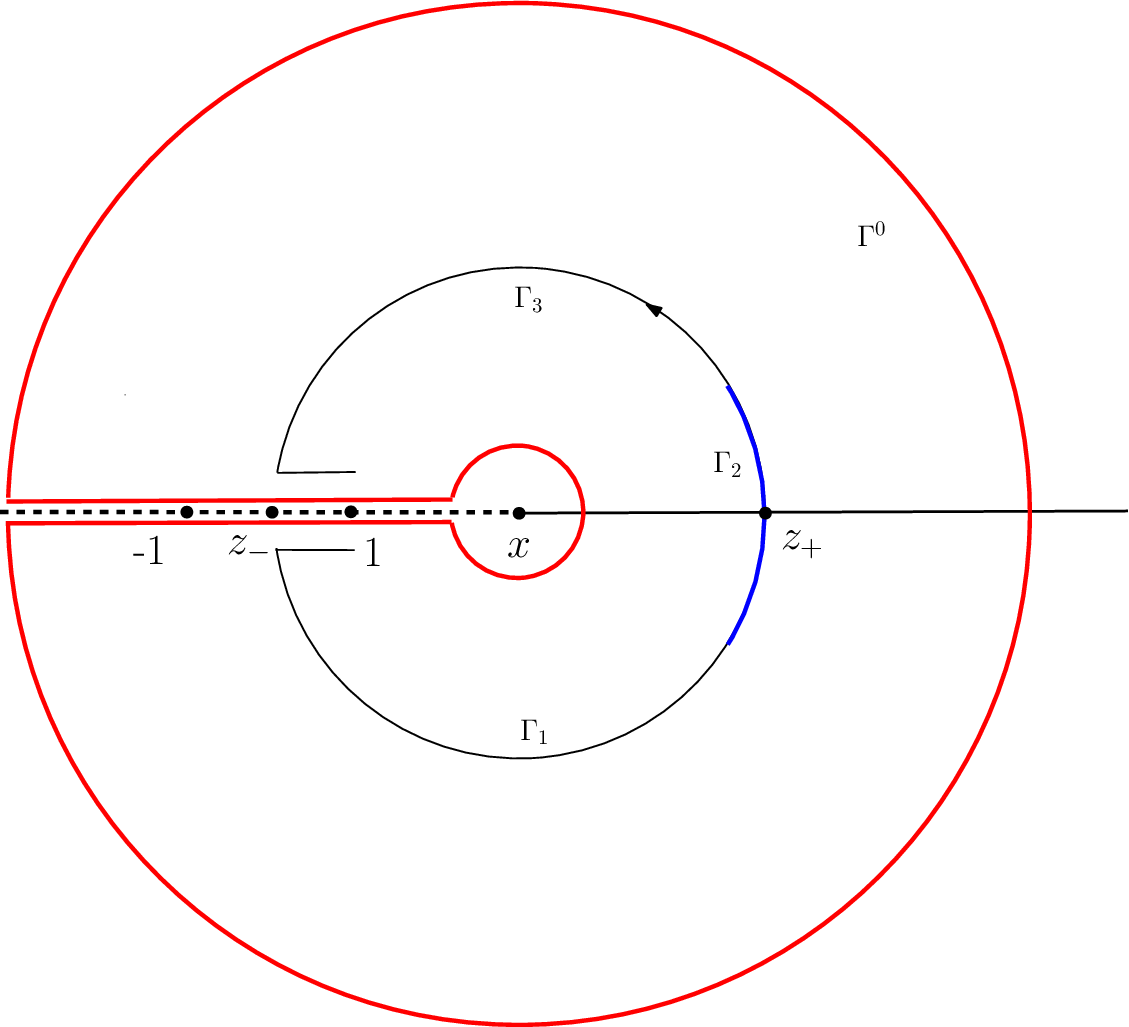}
 \caption{ The path of integration $\Gamma^0$ in \eqref{Eq:ginu}  is colored in red.}    
 \label{Fig:sdcga0} 
\end{figure}

For a more intuitive understanding, we show the behavior of $w(z)$ for $z\in \Gamma_2$ and $z\in \Gamma^0$ in Fig. \ref{Fig:w1ga0}.

\begin{figure}[H]
\centering
\includegraphics[scale=0.6]{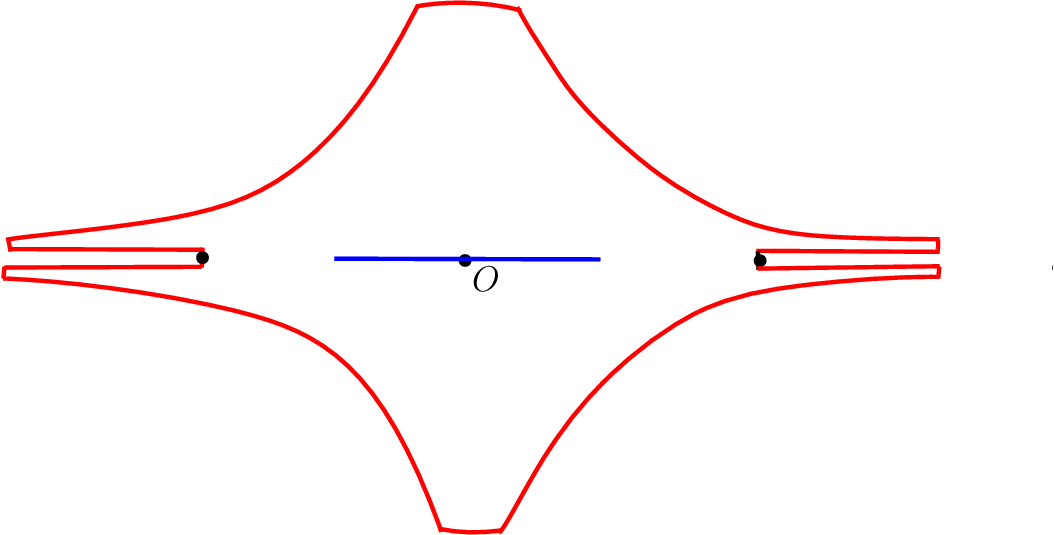}
 \caption{ The contour $\Gamma^0$ and the part of steepest decent contour $\Gamma_2$ in $w$-palne.}    
 \label{Fig:w1ga0}
\end{figure}

\begin{rem}\label{Remf}
   Recalling that $\eta>0$ satisfies \eqref{eta-delta}, we have 
    \[
    f(z)-f(z_+)\geqslant \delta , 
    \]
    for $z\in \Gamma_1\cup \Gamma_3$.
\end{rem}


By substituting  \eqref{Eq:ginu} into \eqref{Eq:J1I2}, we obtain the  double integral expression 
\begin{equation}\label{Eq:I2ju}
I_2=-\frac{e^{-nf(z_+)}}{2^{n+2}\pi^2 }\int_{-\delta}^{\delta}\,e^{-nu^2}\int_{\Gamma^0}\frac{g(z)}{w(z)-u}\,dz\,du,
\end{equation}
where $\Gamma^0$ denotes the closed, clockwise-oriented, contour defined earlier, containing $z_+$ in its interior.
To facilitate numerical analysis, we introduce the coefficients
\begin{equation}\label{Eq:Aj1}
\A_j(e^{\gamma})=-\frac{\Gamma(j+1/2)\sqrt{e^{2\gamma}-1}}{4\pi^{3/2}e^{\gamma}(e^{\gamma}-1)^{\alpha}(e^{\gamma}+1)^{\beta}}\int_{\Gamma^0}\,\frac{g(z)}{w(z)^{2j+1}}dz, 
\end{equation}
which encode the singular behavior of the integrand near $z_+$. Additionally, we define the error term
\begin{equation}\label{Eq:epsilon}
\varepsilon_p(n,\alpha,\beta)=
-\frac{\sqrt{n(e^{2\gamma}-1)}}{4\pi^{3/2}e^{\gamma}(e^{\gamma}-1)^{\alpha}(e^{\gamma}+1)^{\beta}}\int_{-\delta}^{\delta}e^{-nu^2}u^{2p}C_{p}(u)du-2\sqrt{n} \sum_{j=0}^{p-1}I_{j,\delta} \A_j(e^{\gamma}),
\end{equation}
with  $C_{p}(u)$ given by
\begin{equation}\label{Eq:Cpu1}
    C_{p}(u)=\int_{\Gamma^0}\,\frac{g(z)}{w(z)^{2p}(w(z)-u)}dz,
\end{equation}
for $u\in [-\delta, \delta]$, 
and the integral $I_{j,\delta}$ defined by
\begin{align}\label{Eq:Idelta}
I_{j,\delta}=\int_{\delta}^{+\infty}e^{-nu^2}u^{2j}du.
\end{align}
For any non-negative integer $p$, we expand the integrand  \eqref{Eq:I2ju} using the geometric series
\[\frac1{w(z)-u}=\sum_{k=1}^{2p}\frac{u^{k-1}}{w(z)^k}+\frac{u^{2p}}{w(z)^{2p}(w(z)-u)},\]
which allows us to decompose the double integral \eqref{Eq:I2ju} as
\begin{align}\label{Eq:I2juer+}
I_2 &= -\frac{e^{-nf(z_+)}}{2^{n+2}\pi^2 }\left(\sum_{k=1}^{2p}\int_{-\delta}^{\delta}\,e^{-nu^2}u^{k-1}\, du\int_{\Gamma^0}\frac{g(z)}{w(z)^k}\,dz+\int_{-\delta}^{\delta}\,e^{-nu^2}u^{2p}\int_{\Gamma^0}\frac{g(z)}{w(z)^{2p}(w(z)-u)}\,dz\,du\right).
\end{align}
The integrals over $u$ can be evaluated explicitly, and we have
\begin{equation*}
\int_{-\delta}^{\delta}\,e^{-nu^2}u^{k-1}du=
\begin{cases}0, ~~~&k=2j\\
\frac{\Gamma(j+1/2)}{n^{j+1/2}}-2I_{j,\delta},~~&k=2j+1.
\end{cases}
\end{equation*}
Substituting these results into \eqref{Eq:I2juer+} and simplifying, we obtain
\begin{equation}\label{Eq:I2juer1}
 I_2
=\frac{e^{\gamma}(e^{\gamma}-1)^{\alpha}(e^{\gamma}+1)^{\beta}e^{-nf(z_+)}}{2^{n} \sqrt{n\pi(e^{2\gamma}-1)}}\Big\{\sum_{j=0}^{p-1}\frac{\A_j(e^{\gamma})}{n^{j}}+\varepsilon_p(n,\alpha,\beta)\Big\}.
\end{equation}
Note that $g(z_+)=\frac{2e^{\gamma}(e^{\gamma}-1)^{\alpha}(e^{\gamma}+1)^{\beta}}{e^{2\gamma}-1}$. For the special case $p=0$, applying Cauchy integral formula to \eqref{Eq:Aj1} yields
\begin{align*}
\A_0(e^{\gamma})&=-\frac{\Gamma(1/2)\sqrt{(e^{2\gamma}-1)}}{4\pi^{3/2}e^{\gamma}(e^{\gamma}-1)^{\alpha}(e^{\gamma}+1)^{\beta}}\int_{\Gamma^0}\,\frac{g(z)}{[f(z)-f(z_+)]^{\frac{1}{2}}}dz=\frac{1}{\sqrt{e^{2\gamma}-1}}\sqrt{\frac{-2}{f''(z_+)}}=1,
\end{align*}
which confirms the normalization of $\A_0$. 

Note that $(\cosh\gamma-1)^{-\alpha}=2^{\alpha}e^{\alpha\gamma}(e^{\gamma}-1)^{\alpha}$ and $(\cosh\gamma+1)^{-\beta}=2^{\beta}e^{\beta\gamma}(e^{\gamma}+1)^{\beta}$. A combination of \eqref{Eq:Jc1} and \eqref{Eq:I2juer1} gives the asymptotic expansion
\begin{equation}\label{Eq:Ju1}
\begin{split}
P_n^{(\alpha,\beta)}(\cosh\gamma)
=&\frac{e^{(1+\alpha+\beta)\gamma}e^{-nf(z_+)}}{2^{n-\alpha-\beta} (e^{\gamma}-1)^{\alpha}(e^{\gamma}+1)^{\beta}\sqrt{n\pi(e^{2\gamma}-1)}}\\
&\cdot\Big\{\sum_{j=0}^{p-1}\frac{\mathcal{A}_j(e^{\gamma})}{n^{j+1/2}}+\varepsilon_p(n,\alpha,\beta)+\frac{2^{n} \sqrt{n\pi(e^{2\gamma}-1)}}{e^{\gamma}(e^{\gamma}-1)^{\alpha}(e^{\gamma}+1)^{\beta}e^{-nf(z_+)}}(I_1+I_3)\Big\}.
\end{split}
\end{equation}
To establish an upper bound for the remainder term $\varepsilon_p(n,\alpha,\beta)$ defined in \eqref{Eq:epsilon}, we perform a detailed analysis of the three critical components: $C_{p}(u)$ defined in equation \eqref{Eq:Cpu1},  the coefficient $|\A_j(e^{\gamma})|$  specified in \eqref{Eq:Aj1} and the integral term $I_{j,\delta}$ introduced in  \eqref{Eq:Idelta}. 

We first estimate $I_{j,\delta}$ by using integration by parts
\begin{align*}
 I_{j,\delta}=&\int_{\delta}^{+\infty}e^{-nu^2}u^{2j}du\\
=&\sum_{k=1}^j\frac{(2j-1)\cdots(2j-(2(k-1)-1))}{(2n)^k}e^{-n\delta^2}\delta^{2j-(2k-1)}+\frac{(2j-1)\cdots(2j-(2(k-1)-1))\cdots 1}{(2n)^j}\int_{\delta}^{+\infty}e^{-nu^2}du.
\end{align*}
For the last integral, we substitute $u^2=\tau$ and obtain 
\begin{align*}
\int_{\delta}^{+\infty}e^{-nu^2}du=\int_{\delta^2}^{+\infty}\frac{e^{-n\tau}}{2\sqrt{\tau}}d\tau\leqslant \frac{1}{2\delta}\int_{\delta^2}^{+\infty}e^{-n\tau}d\tau\leqslant \frac{e^{-n\delta^2}}{2n\delta}.
\end{align*}
Thus, given $0 < \delta < 1$, we choose $n\geqslant \max\{j,\frac{1}{2\delta}\}$ and have
\begin{align}\label{Eq:erforIde}
|I_{j,\delta}|\leqslant \Big(j\delta+ 1\Big)e^{-n\delta^2}.
\end{align}

We now proceed to estimate $C_{p}(u)$ and $|\A_j(e^{\gamma})|$ as defined in equations \eqref{Eq:Cpu1} and  \eqref{Eq:Aj1}, respectively. This requires a review of the decomposition of the contour $\Gamma^0$:
\[
\Gamma^0=\Gamma_x^0+\Gamma_{(x,o)^{\pm}}+\Gamma_o^0.
\]
Among these components, the critical components for our analysis are $\Gamma_{(x,o)^{\pm}}$, where the behaviors of  $w(z)$, $w(z)-u$ and $g(z)$ require a rigorous examination. Specifically, we must derive lower bounds for $w(z)$, $w(z)-u$ and $g(z)$ along these segments.

By exploiting symmetry in the problem, we focus solely on $\Gamma_{(x,o)^{+}}$ for $|w(z)|$ and $|w(z)-u|$.
To systematically derive the required bounds, we introduce the following notation and partitioning of 
$\Gamma_{(x,o)^+}$.
\begin{nota}
We decompose 
$\Gamma_{(x,o)^+}$ into three distinct segments:
   \begin{align*}
\Gamma_{(x,o)^+}=\Gamma_{(x,o)^+}^1+\Gamma_{(x,o)^+}^2+\Gamma_{(x,o)^+}^3  
   \end{align*}
where
\begin{align*}
   \Gamma_{(x,o)^+}^1=[-10(\cosh\gamma+1)+\cosh\gamma, -1]^+, 
   \end{align*}
   \begin{align*}
       \Gamma_{(x,o)^+}^2=  [-1, 1]^+ 
   \end{align*}
   and 
   \begin{align*}
    \Gamma_{(x,o)^+}^3=[1, \cosh\gamma-\frac{\cosh\gamma-1}{4(\cosh\gamma+2)^4}]^+.
\end{align*} 
\end{nota}

This stratification enables us to isolate distinct regions of the contour where the behavior of $w(z)$, $w(z)-u$ and $g(z)$ can be analyzed individually.

The following lemma provides the necessary lower bounds for $|w(z)|$ and $|w(z)-u|$ when $z\in\Gamma^0$ and $u\in [-\delta, \delta]$, serving as a foundational result for our subsequent analysis.

 \begin{lem}\label{Le:|w||w-u|1}
For $z\in\Gamma^0$ and $u\in [-\delta, \delta]$, where $\delta$ as in equation \eqref{Eq:delta},
we have
 \begin{align}\label{Eq:le|w|}
|w(z)|\geqslant 2\delta
\end{align}
and 
\begin{align}\label{Eq:le|w-u|}
 | w(z)-u| \geqslant \delta.
\end{align}

\end{lem}

The following lemma provides an estimate for $g(z)$, when $z\in\Gamma^0$ and $z\in\Gamma_{11}$. 
\begin{lem}\label{Le:g1}
Assuming that $-1<\alpha\leqslant 0$ and $\beta\leqslant \alpha$, the following estimates for $|g(z)|$ hold for $z$ in different subregions of $\Gamma^0$;
\begin{enumerate}
\item For $z\in\Gamma_x^0$, we have
\begin{equation}\label{Eq:Gax0g}
    |g(z)|\leqslant \frac{2(\cosh\gamma-1)^{\alpha}}{z-\cosh\gamma}.
\end{equation}

\item  For $z\in \Gamma_{(x,o)^+}^1$, we obtain
\begin{equation}\label{Eq:Gaxo1g}
    |g(z)|\leqslant \frac{|z+1|^{\beta}}{\cosh\gamma+1}.
\end{equation}
For $z\in \Gamma_{(x,o)^+}^2$, it yields 
\begin{equation}\label{Eq:Gaxo2g}
  |g(z)|\leqslant \frac{|z-1|^{
  \alpha}|z+1|^{\beta}}{\cosh\gamma-1}.
\end{equation}
For $z\in \Gamma_{(x,o)^+}^3$, one obtains
\begin{equation}\label{Eq:Gaxo3g}
  |g(z)|\leqslant \frac{4(\cosh\gamma+2)^4}{\cosh\gamma-1}|z-1|^{
  \alpha}.
\end{equation}

\item For $z\in \Gamma_{o}^0 $, we have 
\begin{equation}\label{Eq:Gao0g}
  |g(z)|\leqslant \frac{1}{10(\cosh\gamma+1)}.
\end{equation}

\item When $z\in\Gamma_{11}$, we achieve  
\begin{equation}\label{Eq:Ga11g}
  |g(z)|\leqslant \Big(\frac{e^{\gamma}-1}{e^{\gamma}}\Big)^{\alpha}\frac{1}{\sinh \gamma}.
\end{equation}
For $z\in[e^{-\gamma},1]$, it gives 
   \begin{equation}\label{Eq:e-ga1g}
  |g(z)|\leqslant \frac{1}{\cosh\gamma-1}|z-1|^{
  \alpha}.
\end{equation}
\end{enumerate}
   
\end{lem}

Now we are ready to estimate $|C_{p}(u)|$ defined in \eqref{Eq:Cpu1} and $|\A_j(e^{\gamma})|$ in \eqref{Eq:Aj1}.

\begin{lem}\label{Le:Cpo}
The following estimates hold  for $-1<\beta\leqslant \alpha\leqslant 0$.
\begin{enumerate}
    \item For any $u\in [-\delta, \delta]$, we have 
    \begin{align}\label{Eq:Cp}
|C_{p}(u)|
&\leqslant \frac{ 10(\cosh\gamma+2)^5}{(\beta+1)(\cosh\gamma-1)} \cdot\frac{1}{4^p\delta^{2p+1}}.
 \end{align} 
 \item For $j=0,1,\cdots$, we derive
 \begin{align}\label{Eq:erAj}
    |\A_j(e^{\gamma})|\leqslant   \frac{\Gamma(j+1/2)\sqrt{(e^{2\gamma}-1)}}{4\pi^{3/2}e^{\gamma}(e^{\gamma}-1)^{\alpha}(e^{\gamma}+1)^{\beta}}\cdot\frac{ 10(\cosh\gamma+2)^5}{(\beta+1)(\cosh\gamma-1)}\cdot \frac{1}{(2\delta)^{2j+1}}.
 \end{align}

\end{enumerate}

\end{lem}

\begin{proof}
We first prove that 
\begin{equation}\label{Eq:eringg}
    \int_{\Gamma^0}\,| g(z)|dz \leqslant M_g,
\end{equation}
where 
\begin{align}\label{Eq:Mg}
  \notag  M_g=&2\left(\pi +\frac{9}{\beta+1}+\frac{2}{(\beta+1)(\cosh\gamma-1)}+\frac{ 4(\cosh\gamma+2)^4(\cosh\gamma-1)^{\alpha}}{\alpha+1} +2 \pi (\cosh\gamma-1)^{\alpha}\right)\\
  \leqslant &\frac{ 10(\cosh\gamma+2)^5}{(\beta+1)(\cosh\gamma-1)}.  
\end{align}

By inequality \eqref{Eq:Gax0g},  we obtain 
\begin{align*}
|\int_{\Gamma_{x}^0}g(z)dz|\leqslant 2(\cosh\gamma-1)^{\alpha}\left|\int_{\Gamma_{x}^0}\frac{1}{z-\cosh\gamma}dz\right|.
\end{align*}
Since  $\Gamma_x^0=\left\{z| \, |z-\cosh\gamma|=\frac{\cosh\gamma-1}{4(\cosh\gamma+2)^4} \right \}$, it follows that
\begin{align}\label{Eq:gx}
|\int_{\Gamma_{x}^0}g(z)dz|\leqslant 4 \pi (\cosh\gamma-1)^{\alpha}.
\end{align}
 Applying inequality \eqref{Eq:Gaxo1g} and recalling $-1<\beta\leqslant 0$, we have
 \begin{align}\label{Eq:gox1}
\int_{\Gamma_{(x,o)^+}^1}|g(z)||dz|\leqslant \frac{1}{\cosh\gamma+1}|\int_{\Gamma_{(x,o)^+}^1}(z+1)^{\beta}dz|
\leqslant \frac{(9(\cosh\gamma+1))^{\beta+1}}{(\beta+1)(\cosh\gamma+1)}<\frac{9}{\beta+1}.
\end{align}
Note that $1\leqslant |z-1|\leqslant 2$ when $z\in [-1, 0]$ and $1\leqslant |z+1|\leqslant 2$ when $z\in (0, 1)$. Recalling $-1<\alpha\leqslant 0$ and $\beta\leqslant \alpha$, we have $\max \{2^{\alpha}, 1, 2^{\beta}\}=1$.
 It is easy to see from inequality \eqref{Eq:Gaxo2g} that 
\begin{align}\label{Eq:gox2}
\int_{\Gamma_{(x,o)^+}^2}|g(z)|dz\leqslant& \frac{1}{\cosh\gamma-1}\int_{-1}^{1}(1-z)^{\alpha}(z+1)^{\beta}dz
\leqslant\frac{1}{\cosh\gamma-1}\Big(\frac{1}{\beta+1}+\frac{1}{\alpha+1}\Big)\leqslant\frac{2}{(\beta+1)(\cosh\gamma-1)}.
\end{align}

From inequality \eqref{Eq:Gaxo3g}, we obtain
\begin{align}\label{Eq:gox3}
\notag\int_{\Gamma_{(x,o)^+}^3}\, |g(z)|dz\leqslant&\frac{4(\cosh\gamma+2)^4}{\cosh\gamma-1}\left|\int_{\Gamma_{(x,o)^+}^3}(z-1)^{\alpha}dz\right|\\
\leqslant & \frac{4(\cosh\gamma+2)^4}{\cosh\gamma-1}\frac{(\cosh\gamma-1-\frac{\cosh\gamma-1}{4(\cosh\gamma+2)^4})^{\alpha+1}}{\alpha+1}\leqslant \frac{ 4(\cosh\gamma+2)^4(\cosh\gamma-1)^{\alpha}}{\alpha+1}.
\end{align}
By symmetry, we have from inequalities \eqref{Eq:gox1}-\eqref{Eq:gox3} that
\begin{align}\label{Eq:gxo+-}
\int_{\Gamma_{(x,o)^\pm}}\, |g(z)|dz\leqslant  2\left(\frac{9}{\beta+1}+ \frac{2}{(\beta+1)(\cosh\gamma-1)}+\frac{ 4(\cosh\gamma+2)^4(\cosh\gamma-1)^{\alpha}}{\alpha+1}\right).
\end{align}
One may directly check from Lemma \ref{Le:|w||w-u|1} and inequality \eqref{Eq:Gao0g}  that
\begin{equation}\label{Eq:go}
\int_{\Gamma_o^0}\,|g(z)|dz
\leqslant 2\pi.
\end{equation}
 A combination of $\Gamma^0=\Gamma_x^0+\Gamma_{(x,o)^{\pm}}+\Gamma_o^0$ with inequalities \eqref{Eq:gx}, \eqref{Eq:gxo+-} and \eqref{Eq:go} yields inequality \eqref{Eq:eringg}.

1. Now we estimate the error of $C_p(u)$. The equation \eqref{Eq:Cpu1} and Lemma \ref{Le:|w||w-u|1} yields that 
\begin{equation*}
    C_{p}(u)\leqslant \frac{1}{4^p\delta^{2p+1}}\int_{\Gamma^0}\,|g(z)|dz.
\end{equation*}
Therefore, inequality \eqref{Eq:eringg} yields that 
\[
   C_{p}(u)\leqslant \frac{1}{4^p\delta^{2p+1}}M_g.
\]
The inequality \eqref{Eq:Cp} follows from inequality \eqref{Eq:Mg}.

2. It follows from equation \eqref{Eq:Aj1} and Lemma \ref{Le:|w||w-u|1} that
\begin{equation*}
\A_j(e^{\gamma})=\frac{\Gamma(j+1/2)\sqrt{(e^{2\gamma}-1)}}{4\pi^{3/2}e^{\gamma}(e^{\gamma}-1)^{\alpha}(e^{\gamma}+1)^{\beta}}\cdot \frac{1}{(2\delta)^{2j+1}}\int_{\Gamma^0}\,|g(z)|dz. 
\end{equation*}
Therefore, from inequalities \eqref{Eq:eringg} and \eqref{Eq:Mg}, equation \eqref{Eq:erAj} hold.
\end{proof}

\subsection{Estimation of \( |I_1| \) and \( |I_3| \)}
We now proceed to estimate the integral $|I_1|$.
For $z\in[e^{-\gamma},1]$, we first establish the inequality
 \begin{equation}\label{Eq:f-fz+}
     f(z)-f(z_+)\geqslant f(z_-)-f(z_+)=2\gamma.
 \end{equation}
By substituting the expression for $I_1$	
  in \eqref{Eq:Ij1} and invoking the definition of $\Gamma_1$ in \eqref{Eq:Gam1}, we derive the result 
\begin{align*}
&\left|\frac{2^{n} \sqrt{n\pi(e^{2\gamma}-1)}}{e^{\gamma}(e^{\gamma}-1)^{\alpha}(e^{\gamma}+1)^{\beta}e^{-nf(z_+)}}I_1\right|=\left|\frac{\sqrt{n(e^{2\gamma}-1)}}{2\sqrt{\pi}e^{\gamma}(e^{\gamma}-1)^{\alpha}(e^{\gamma}+1)^{\beta}}  \int_{\Gamma_1}\,e^{-n[f(z)-f(z_+)]}g(z)\,dz\right|\\\notag
\leqslant &\frac{\sqrt{n(e^{2\gamma}-1)}}{2\sqrt{\pi}e^{\gamma}(e^{\gamma}-1)^{\alpha}}\left(\left|\int_{\Gamma_{11}}\,e^{-n[f(z)-f(z_+)]}g(z)\,dz\right|+\left|\int_{e^{-\gamma}}^1\,e^{-n[f(z)-f(z_+)]}g(z)\,dz\right| \right).  
\end{align*}
Applying Remark \ref{Remf} and inequality \eqref{Eq:Ga11g} to the first term of the right-hand side of the above inequality and inequalities \eqref{Eq:e-ga1g} and \eqref{Eq:f-fz+} to the second term of the right-hand side of the above inequality, we obtain the  result
\begin{align*}
\left|\frac{2^{n} \sqrt{n\pi(e^{2\gamma}-1)}}{e^{\gamma}(e^{\gamma}-1)^{\alpha}(e^{\gamma}+1)^{\beta}e^{-nf(z_+)}}I_1\right|\leqslant& \frac{\sqrt{\pi n(e^{2\gamma}-1)}e^{-n\delta}}{2e^{(\alpha+1)\gamma}\sinh\gamma }+\frac{\sqrt{n(e^{2\gamma}-1)}e^{-2n\gamma-\gamma}}{2\sqrt{\pi}(e^{\gamma}-1)^{\alpha}(\cosh\gamma-1) }  \int_{e^{-\gamma}}^1 |z-1|^\alpha \,dz. 
\end{align*}
Straightforward calculation then gives
\begin{align}\label{Eq:I1}
\left|\frac{2^{n} \sqrt{n\pi(e^{2\gamma}-1)}}{e^{\gamma}(e^{\gamma}-1)^{\alpha}(e^{\gamma}+1)^{\beta}e^{-nf(z_+)}}I_1\right|\leqslant \frac{\sqrt{\pi n}e^{-n\delta+\gamma}}2+\frac{\sqrt{n}e^{-2n\gamma}}{(\alpha+1)(\cosh\gamma-1) }.
 \end{align}

By the symmetry in the contour integration, the bound for $|I_3|$ follows identically.

\subsection{ Estimation of \(P_n^{(\alpha,\beta)}(\cosh\gamma)\)}
Considering the error contribution from $I_2$ and utilizing the estimates for $I_1$ and $I_3$, we now proceed to establish the error bound for the Jacobi polynomial approximation $P_n^{(\alpha,\beta)}(\cosh\gamma)$.

\begin{lem}\label{Le:epsilon} 
Let $\varepsilon_p(n,\alpha,\beta)$, $I_1$ and $I_3$ be defined as in the previous sections. We introduce the following quantities
\[
\zeta_p(n,\alpha,\beta):=\varepsilon_p(n,\alpha,\beta)+\frac{2^{n} \sqrt{n\pi(e^{2\gamma}-1)}}{e^{\gamma}(e^{\gamma}-1)^{\alpha}(e^{\gamma}+1)^{\beta}e^{-nf(z_+)}}(I_1+I_3)
\]
and 
\begin{align}\label{Eq:tcp}
c_p:=\left(1+\frac{1}{4}\left( p(p-1)\delta +2p+2\right)\left(\frac{p+1}{\delta^2 e}\right)^{p+1}\right)\frac{e^{2\gamma}}{4^p\delta^{2p+1}}\cdot\frac{ 10(\cosh\gamma+2)^5}{(\beta+1)(\cosh\gamma-1)}.
\end{align}
Then, for any integer $p\geqslant 1$, the following inequality holds 
\begin{align*}
|\zeta_p(n,\alpha,\beta)|\leqslant\frac{c_p\Gamma(p+\frac12)}{n^p}.
\end{align*}

\end{lem}
\begin{proof}
Let us define the constant  $\tilde{c}_p$ as follows
\begin{align*}
\tilde c_p:=\frac{e^{2\gamma}}{4^p\delta^{2p+1}}\cdot\frac{ 10(\cosh\gamma+2)^5}{(\beta+1)(\cosh\gamma-1)}.
\end{align*}
This constant arises naturally from Lemma \ref{Le:Cpo} and the subsequent analysis. Utilizing this definition, we first establish bounds for the coefficients $C_p(u)$ and $\A_j(e^{\gamma})$
\begin{align*}
|C_p(u)|\leqslant e^{-2\gamma}\tilde c_p  \qquad \text{and} \qquad |\A_j(e^{\gamma})|\leqslant   \frac{e^{-2\gamma}\Gamma(j+1/2)\sqrt{e^{2\gamma}-1}}{16\pi^{3/2}e^{\gamma}(e^{\gamma}-1)^{\alpha}(e^{\gamma}+1)^{\beta}}\tilde{c}_j.
\end{align*}
These inequalities follow directly from Lemma \ref{Le:Cpo}. 
Next, we substitute these bounds into the expression for $\varepsilon_p(n,\alpha,\beta)$ given in \eqref{Eq:epsilon} 
\begin{align*}
|\varepsilon_p(n,\alpha,\beta)|&\leqslant \frac{e^{-2\gamma}\sqrt{e^{2\gamma}-1}}{4\pi^{3/2}e^{\gamma}(e^{\gamma}-1)^{\alpha}(e^{\gamma}+1)^{\beta}}\left|\tilde c_p\sqrt{n}\int_{-\infty}^{\infty}e^{-nu^2}u^{2p}du+\frac{\sqrt{n}}{2} \sum_{j=0}^{p-1} I_{j,\delta} \Gamma(j+1/2)\tilde{c}_j\right|.
\end{align*}
Applying the inequality
\(
\frac{e^{-2\gamma}\sqrt{e^{2\gamma}-1}}{4\pi^{3/2}e^{\gamma}(e^{\gamma}-1)^{\alpha}(e^{\gamma}+1)^{\beta}}<1
\)
for $-1< \beta\leqslant\alpha\leqslant 0 $, we simplify the bound to
\begin{align*}
|\varepsilon_p(n,\alpha,\beta)|
\leqslant\frac{\tilde c_p \Gamma(p+\frac12)}{n^p}+\frac{\tilde c_p \Gamma(p+\frac12)\sqrt{n}}{2}\sum_{j=0}^{p-1} |I_{j,\delta}|.
\end{align*}
To estimate $\zeta_p(n,\alpha,\beta)$, we combine the above result with inequalities \eqref{Eq:erforIde} and \eqref{Eq:I1}. The result is
\begin{align*}
|\zeta_p(n,\alpha,\beta)|&\leqslant \frac{\tilde c_p\Gamma(p+\frac12)}{n^p}+\frac{\tilde c_p \Gamma(p+\frac12)\sqrt{n}e^{-n\delta^2}}{2}\sum_{j=0}^{p-1}  (j\delta+ 1)+\sqrt{\pi n}e^{-n\delta+\gamma}+\frac{2\sqrt{n}e^{-2n\gamma}}{(\alpha+1)(\cosh\gamma-1) }.
\end{align*}
From the definition of $\delta$, we have $\delta^2\leqslant \max\{\delta,2\gamma\}$.
Therefore, a simple calculation gives
\begin{align*}
|\zeta_p(n,\alpha,\beta)|&\leqslant  \frac{c_p(n)\Gamma(p+\frac12)}{n^p},
\end{align*}
where
\[
c_p(n)=\tilde{c}_p+\left( \frac{\tilde{c}_p}{2}\left(\frac{p(p-1)\delta}{2}+p\right)+\frac{\sqrt{\pi}e^{\gamma}}{\Gamma(p+\frac{1}{2})} +\frac{2}{(\alpha+1)(\cosh\gamma-1) \Gamma(p+\frac{1}{2})} \right)l_p(n)
\]
and 
\[
l_p(n)=n^{p+1}e^{-n\delta^2}.
\]
The function $l_p(n)$ attains its maximum at $n=\frac{p+1}{\delta^2}$
as verified by setting $l_p'(n)=0$. Substituting this critical point and using $\tilde c_p\geqslant\max\{10e^{\gamma},\frac{10}{(\alpha+1)(\cosh\gamma-1)}\}$, we obtain
\[
c_p(n)\leqslant \left(1+\frac{1}{4}\left( p(p-1)\delta +2p+2\right)\left(\frac{p+1}{\delta^2 e}\right)^{p+1}\right)\tilde{c}_p,
\]
for any $n\geqslant 1$. 

This completes the proof. 
\end{proof}

 Now we improve the error estimate to 
\begin{align*}
|\zeta_p(n,\alpha,\beta)|\leqslant \frac{\A_p(e^{\gamma})}{n^{p}}+\frac{c_{p+1}\Gamma(p+\frac32)}{n^{p+1}}= \frac{\hat c_p}{n^p},
\end{align*}
valid for $p\geqslant 1$, where the coefficient
$\hat c_p$  is defined by
\[
\hat c_p=|\A_p(e^{\gamma})|+\left|\frac{c_{p+1}\Gamma(p+\frac32)}{n}\right|.
\]
The constant $ c_{p+1}$ follows from equation \eqref{Eq:tcp}.

We recall the formula
\[ -f(z_+)=\gamma+\log 2 \]
 and the asymptotic expansion derived in equation \eqref{Eq:Ju1}, which establish the main result stated in Theorem \ref{Th:C1}.

\section{Case 2: $0\leqslant x<1$}\label{Sec:C2}
 Assume $x=\cos \gamma$ with $\gamma\in (0, \frac{\pi}2]$.  We rewrite equation \eqref{Eq:Jacobi1} as follows 
\begin{equation}\label{Eq:Ja2}
P_n^{(\alpha,\beta)}(x)=\frac{(-1)^n (1-x)^{-\alpha}(1+x)^{-\beta}}{2^{n+1}\pi i}\int_{\mathcal L} e^{-nf(z)}g(z)dz,
\end{equation}
where 
$f(z)$ is the phase function in equation \eqref{Eq:fz2} 
and 
\begin{equation}\label{Eq:gz2}
g(z)=\frac{(1-z)^{\alpha}(z+1)^{\beta}}{z-x}.
\end{equation}

Since the saddle points \(
z_{\pm}=x\pm i\sqrt{1-x^2}=e^{\pm i\gamma}
\), by selecting the principal branch of the logarithm with a branch cut along $(-\infty, x]\cup [1, +\infty)$, we evaluate $f(z)$ at $z_{\pm}$
\begin{align}\label{Eq:fzpm}
f(z_\pm)=-\log2\pm i(\pi -\gamma).
\end{align}
Next, we compute the second derivative of $f(z)$
\[
f''(z)=-\frac{1}{(z-x)^2}+\frac{2(z^2+1)}{(z^2-1)^2}.
\]
Evaluating at $z_{\pm}$, we find
\[
f''(z_{\pm})=\frac{ e^{\pm (\frac{\pi}2-\gamma)i}}{\sin\gamma}.
\]

Due to symmetry, we focus on the steepest decent contour passing through $z_+$. To describe this contour, we introduce the analytic function
\begin{equation}\label{Eq:wz2}
w(z)=[f(z)-f(z_+)]^{1/2}=\Big[\log\frac{z-x}{1-z^2}+\log2-(\pi -\gamma)i\Big]^{1/2}, 
\end{equation}
where the square root is defined with a branch cut ensuring analyticity in a neighborhood of $z_+$.
As $z\to z_+$, we perform a Taylor expansion 
\begin{align*}
f(z)=f(z_+)+\frac{f''(z_+)}{2}(z-z_+)^2+O\big((z-z_+)^3\big).
\end{align*}
Based on the local behavior near $z_+$, we select the square root function in \eqref{Eq:wz2} to maintain analyticity. Specifically, we choose
\begin{align*}
w(z)=\frac{e^{\pm(\frac{\pi }4-\frac{\gamma}2)i}}{\sqrt{2\sin\gamma}}(z-z_+)+O\big((z-z_+)^2\big).
\end{align*}

Following an approach analogous to Case 1, we examine the contour $ABCDEFGHIJKLMN$ shown in Fig. \ref{Fig:zcon2}. This contour is symmetric about the real axis, reflecting the parity properties of the original function $f(z)$. Given that the saddle point $z_+$ plays a central role, the image of its neighboring area will be duplicated in the $z$-plane. To prevent multi-valuedness, we partition the upper half of the $z$-plane into two distinct regions that pass through 
$z_+$. The corresponding images $f(z)$ and $w(z)$ of these upper half regions in the 
$z$-plane  are illustrated in  Figs. \ref{Fig:f2} and \ref{Fig:w2}.

\begin{figure}[H]
\center
\includegraphics[scale=0.7] {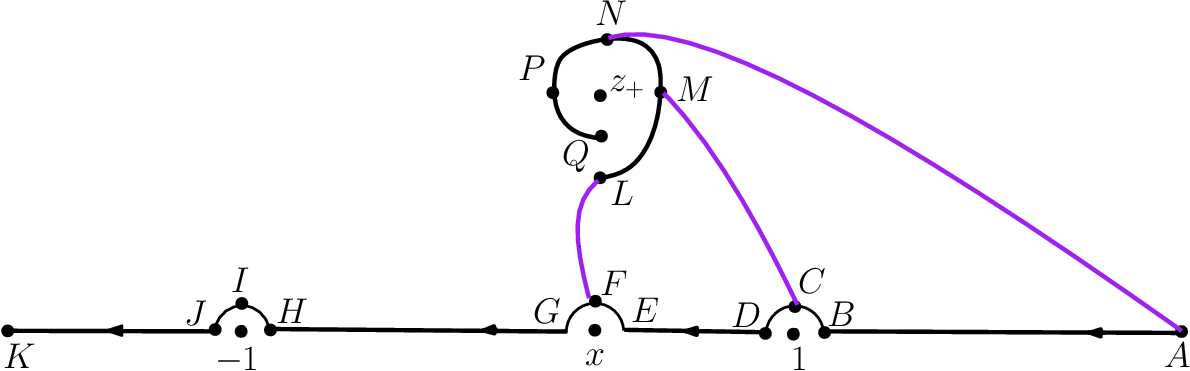}
 \caption{ The $z-$plane.}    
  \label{Fig:zcon2} 
\end{figure}
\begin{figure}[H]
\center
\includegraphics[scale=0.6]  {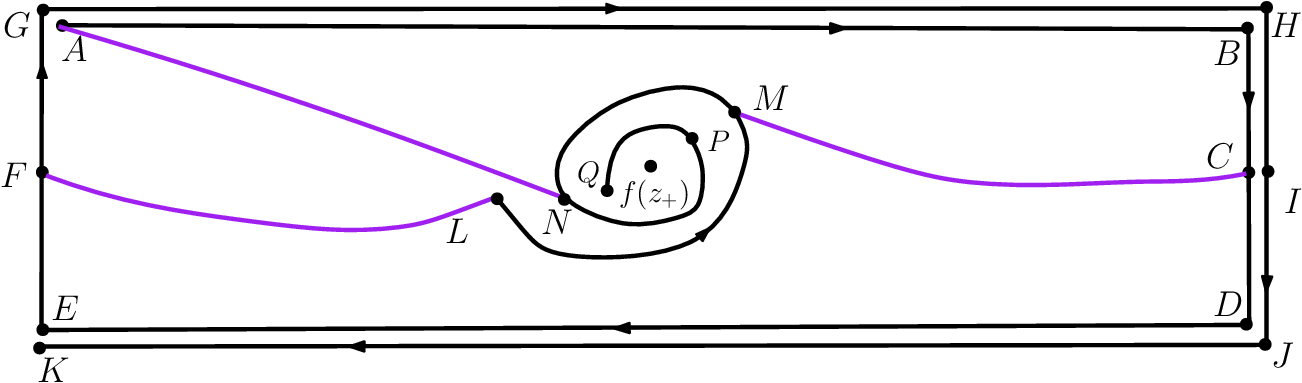}
 \caption{ The $f-$plane.}     
 \label{Fig:f2}
\end{figure}
\begin{figure}[H]
\center
\includegraphics[scale=0.55] {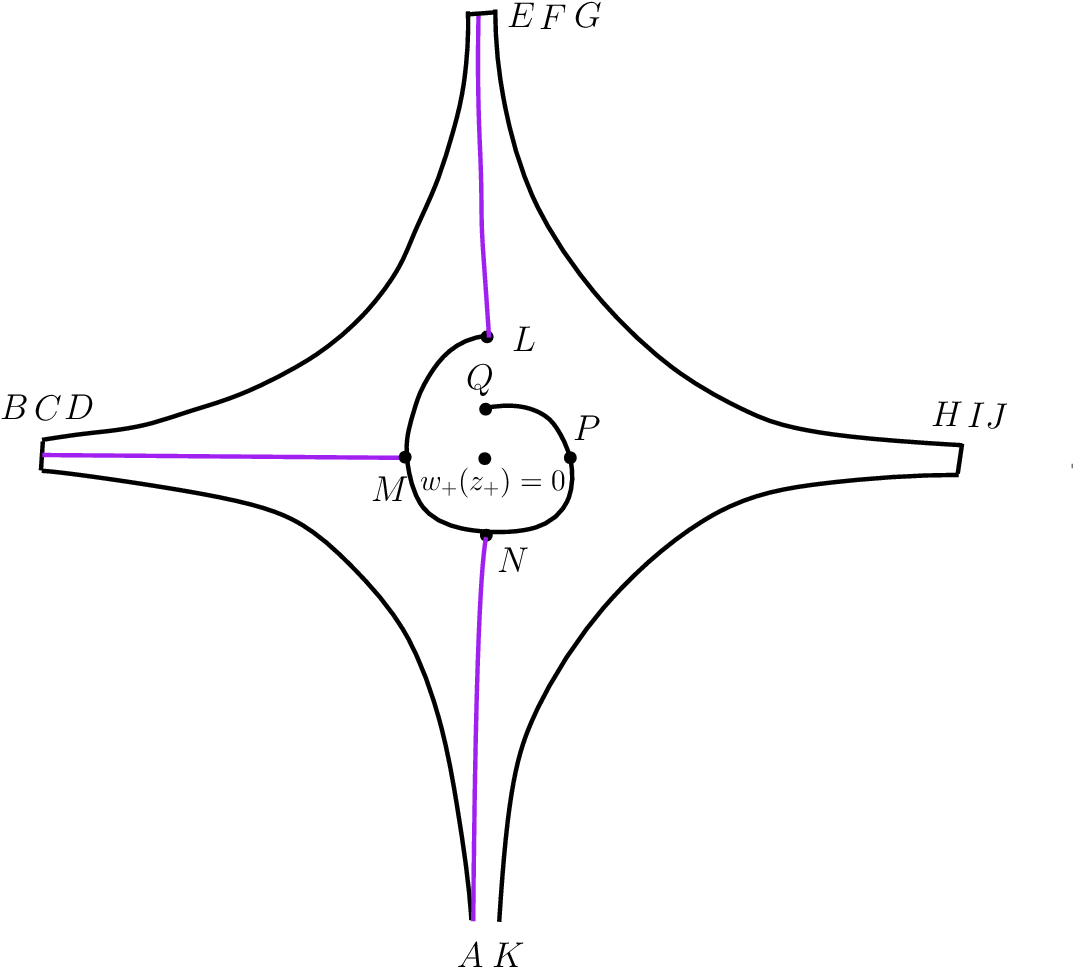}
 \caption{ The $w-$plane.}     
 \label{Fig:w2}
\end{figure}

The arguments of $f(z)$ and $w(z)$ at $A,\,B,\,C,\,D,  E, \, F, \, G, \, H, \, I,\, J, \, K, \, L, \, M, \, N, \, P, \, Q$ can be the ones given in Tab. \ref{Tab:argfw}.
\begin{table}[H]
\centering
\begin{tabular}{ |ccccccccccc| }
	\hline
	{ ~~~ } & A  &B\,C\,D & E\,F\,G &H\,I\,J &K	 & L & M & N & P & Q  \\ \hline
	{ $\arg f$} & $3 \pi$ & $2\pi $ &$\pi$ & 0 & $-\pi$ &$\pi$ & $2\pi$ & $3\pi$ & $4\pi$ & $5\pi$\\ \hline
	{ $\arg w$} & $\frac{3 \pi}2$ & $\pi $ &$\frac{\pi}2$ & 0 & $-\frac{\pi}2$ &$\frac{\pi}2$ & $\pi$ & $\frac{3\pi}2$ & $2\pi$ & $\frac{5\pi}2$ \\ \hline
\end{tabular}
 \caption{ The argument of $f(z)$ and $w(z)$ }
 \label{Tab:argfw}
\end{table}
We will use this table to calculate the modulus of $w(z)$  and the imaginary parts of $w(z)$ in some segments of the $z$-plane.

 For seeking the steepest decent curve $\Gamma^+_{\varepsilon}$ (resp. $\Gamma^-_{\varepsilon}$) through $z_+$ (resp. $z_-$), we need to solve the equation
\begin{align*}
f(z)-f(z_+)=\kappa\geqslant0.
\end{align*}
Recall that $x=\cos \gamma$ with $\gamma\in (0, \frac{\pi}2]$. It follows from equations \eqref{Eq:fz2}  and \eqref{Eq:fzpm} that
\begin{align}\label{Eq:z}
\frac{z-\cos\gamma}{1-z^2}=-\frac{e^{\kappa-i\gamma}}{2}=-\frac{1}{2s},
\end{align}
where $s=e^{-\kappa+i\gamma}$.  Expressing $s$ in polar form as $s=r e^{i \gamma}$, with $0\leqslant r\leqslant 1$, we establish the inequality
\begin{align}\label{Eq:s-x}
\sin\gamma|\cos\gamma|\leqslant |s-x|\leqslant \max\{\sin\gamma,1-\cos\gamma\}\leqslant 1.
\end{align}
Solving equation \eqref{Eq:z} yields the parametric solution for $z$
\begin{align}\label{Eq:zs}
z=re^{i\gamma}\pm \sqrt{(re^{i\gamma}-e^{i\gamma})(re^{i\gamma}-e^{-i\gamma})}=re^{i\gamma}\pm \sqrt{(r-1)(re^{2\gamma i}-1)}.
\end{align}
Expressing $z=a+ib$, we substitute it into the first formula of \eqref{Eq:zs} to obtain two distinct solutions
 \begin{equation*}
 \begin{cases}
(a-r\cos\gamma)^2-(b-r\sin\gamma)^2=(1-r)(1-r\cos2\gamma)\\
(a-r\cos\gamma)(b-r\sin\gamma)=-\frac{r(1-r)\sin2\gamma}2.
 \end{cases}
 \end{equation*}
Solving these equations provides the parametric expressions for $a$ and $b$ in terms of $r$
 \begin{equation*}
 \begin{cases}
a=r\cos\gamma\pm\sqrt{\frac{(1-r)(1-r\cos2\gamma+\sqrt{r^2+1-2r\cos2\gamma})}2}\\
b=r\sin\gamma\mp\sqrt{\frac{(1-r)(r\cos2\gamma-1+\sqrt{r^2+1-2r\cos2\gamma})}2}.
 \end{cases}
 \end{equation*}

Now we can represent the steepest decent curve $\Gamma^+_{\varepsilon}$  (see Fig. \ref{steepest_decent2}), which encloses the point $x$ and whose endpoints approach $-1$ and $1$ as $\varepsilon\to 0$. 

\begin{figure}[H]
\centering
\includegraphics[scale=0.6] {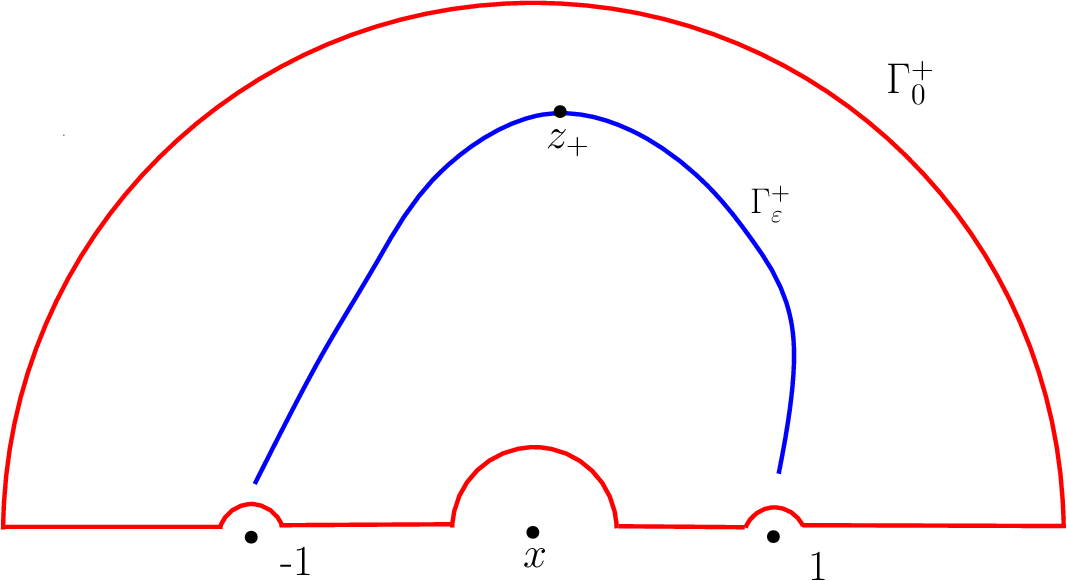}
 \caption{ The steepest decent contour $\Gamma^+_{\varepsilon}$, and contour $\Gamma_0^+$ enclosing it.}  
 \label{steepest_decent2}
\end{figure}

\begin{lem}\label{Lem:zz-x}
    For all $z\in\Gamma^+_{\varepsilon}$, we obtain the following inequalities 
    \begin{equation}\label{Eq:|z|le1}
       |z|\leqslant 1+\sqrt{2}  
    \end{equation}
    and
    \begin{align}\label{Eq:|z-x|}
    |z-\cos\gamma|\geqslant \frac{\sin^2\gamma}{3}.
\end{align}
\end{lem}

For $ 0<\gamma<\frac\pi{2}$, we keep in main that    
\begin{align}\label{Eq:12singa}
    \frac12\sin\gamma\leqslant \sin\frac\gamma{2}.
\end{align} 
Choosing $0<\varepsilon<2\sin^2\frac\gamma{2}-\frac{\sin^2\gamma}{8}$, a direct computation yields
\begin{align*}
\cos\gamma+\frac{\sin^2\gamma}{8}=1-2\sin^2\frac\gamma{2}+\frac{\sin^2\gamma}{8}<1-\varepsilon. 
\end{align*}

 For convenience, we introduce the following notations.
\begin{nota}\label{No:Gam0+}
We define the contour $\Gamma_0^+$ as the union of the following  eight arcs and line segments, with $\gamma$ defined as before:
\begin{align*}  
\Gamma_{o}^+&=\bigg \{z|z=\cos\gamma+4e^{1+i\theta},~\theta \in (0,\pi)\bigg\};\\
\Gamma_{(o,1)}^+&=\bigg \{z|z=1+\mu, ~\mu \in [\varepsilon,\cos\gamma+4e-1]^+\bigg\};\\
\Gamma_{1}^+&=\bigg \{z|z=1+\varepsilon e^{i\theta}, ~\theta \in \big(0,\pi)\bigg\};\\
\Gamma_{(1,x)}^+&=\bigg \{z|z=\cos\gamma+\mu, ~\mu \in [\frac{\sin^2\gamma}{8},1-\cos\gamma-\varepsilon]^+\bigg\};\\
\Gamma_x^+&=\bigg \{z|z=\cos\gamma+\frac{\sin^2\gamma}{8}e^{i\theta}, ~\theta \in \big(0,\pi\big)\bigg\}
;\\
\Gamma_{(x,-1)}^+&=\bigg \{z|z=\cos\gamma+\mu, ~\mu \in [-1-\cos\gamma+\varepsilon, -\frac{\sin^2\gamma}{8}]^+\bigg\};\\
\Gamma_{-1}^+&=\bigg \{z|z=-1+\varepsilon e^{i\theta}, ~\theta \in \big(0,\pi)\bigg\};\\
\Gamma_{(-1,o)}^+&=\bigg \{z|z=-1+\mu, ~\mu \in [\cos\gamma-4e+1, 
 -\varepsilon]^+\bigg\}.
\end{align*}
Here, $\varepsilon$  satisfies 
\begin{align}\label{Eq:var2}
0<\varepsilon<\min\{2\sin^2\frac\gamma{2}-\frac{\sin^2\gamma}{8}, e^{-32}, \, \sqrt{1-\cos\gamma},\, \frac{1-\cos\gamma}{18e^{\pi-\gamma}}\},
\end{align}
and tends to $0$. 
\end{nota}

Combining these definitions with Lemma \ref{Lem:zz-x}, we observe that $\Gamma_0^+$ forms a closed contour enclosing $\Gamma_\varepsilon^+$ in its interior (see Fig. \ref{steepest_decent2}).

Applying a similar approach as in \cite{SNWW23}, we analyze the behavior of $w(z)$ as $z$ traverses from $1$ to $-1$ in the upper half-plane along the contour $\Gamma^+_{\varepsilon}$
(see Fig. \ref{Fig:zcon2}). During this traversal,  $w(z)$ moves from $-\infty$ to $+\infty$ as shown in Fig. \ref{Fig:sd2}.

\begin{figure}[H]
\center
\includegraphics[scale=0.7] {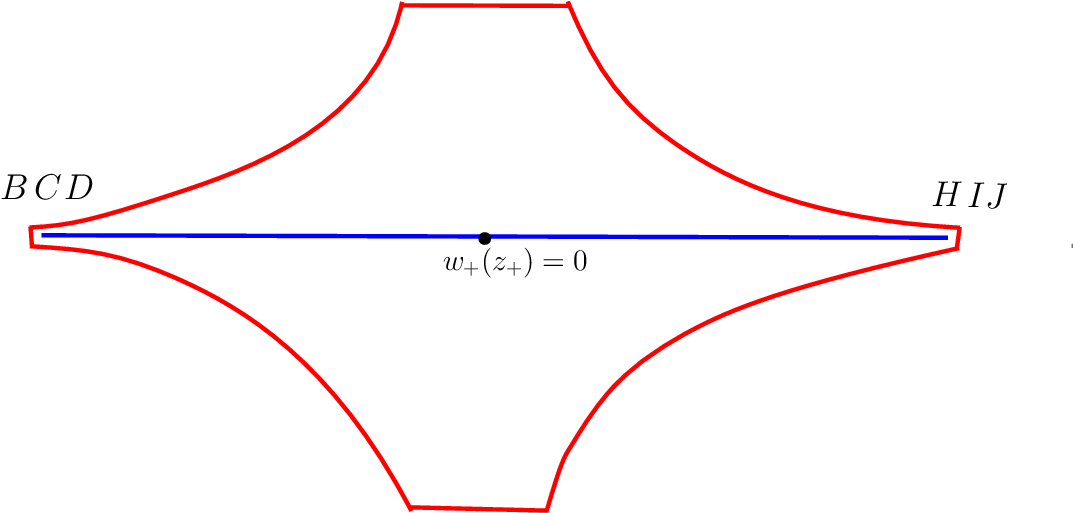}
 \caption{ The contour $\Gamma^0$ and the steepest decent $\Gamma_\varepsilon^+$ in $w_+$-palne.}     
 \label{Fig:sd2}
\end{figure}

We now deform the contour $\mathcal L$ to the steepest decent contour $\Gamma^+_\varepsilon \cup \Gamma^-_\varepsilon$, This deformation allows us to rewrite equation \eqref{Eq:Ja2} as
\begin{equation}\label{Eq:Jc2}
\begin{split}
P_n^{(\alpha,\beta)}(x)
=&(1-x)^{-\alpha}(1+x)^{-\beta}\Big(\frac{(-1)^ne^{-nf(z_+)}}{2^{n+1}\pi i}\lim_{\varepsilon\to0}\int_{\Gamma^+_\varepsilon}\,e^{-n[f(z)-f(z_+)]}g(z)\,dz\\
&+\frac{(-1)^ne^{-nf(z_-)}}{2^{n+1}\pi i}\lim_{\varepsilon\to0}\int_{\Gamma^-_\varepsilon}\,e^{-n[f(z)-f(z_+)]}g(z)\,dz\Big)\\
=&:(1-x)^{-\alpha}(1+x)^{-\beta}\left(I_++I_-\right),
\end{split}
\end{equation}
where $I_+$ and $I_-$ represent the integrals over $\Gamma^+_\varepsilon$ and $\Gamma^-_\varepsilon$, respectively. Due to the symmetry of the steepest descent contours $\Gamma_{\varepsilon}^+$ and $\Gamma_{\varepsilon}^-$ with respect to the real axis, and the property $f(z)=\overline {f(\bar z)}$, the integrals $I_+$ and $I_-$ are complex conjugates. Therefore, it suffices to compute $I_+$.

We introduce a new variable $u$, where $u=w(z)$. For  $z\in \Gamma^+_\varepsilon $, the variable $u$  ranges over the interval  
$ (-\frac{ \sqrt{-\log \varepsilon}}2, \frac{ \sqrt{-\log \varepsilon}}2)$. 
Substituting this into the integral, we can express $I_+$ as 
\begin{equation*}\label{Eq:I+}
I_+
=\frac{(-1)^{n}e^{-nf(z_+)}}{2^{n+1}\pi i }\lim_{\varepsilon\to0}\int_{-\frac{ \sqrt{-\log \varepsilon}}2}^{\frac{ \sqrt{-\log \varepsilon}}2}\,e^{-nu^2}\frac{g(w^{-1}(u))}{w'(w^{-1}(u))}\,du.
\end{equation*}
Applying the Cauchy integral formula 
\[
\frac{g(w^{-1}(u))}{w'(w^{-1}(u))}=\frac1{2\pi i}\int_{\Gamma_0^+}\frac{g(z)}{w(z)-u}\,dz, 
\]
 we obtain the following double integral   
\begin{equation}\label{Eq:Ijc2}
I_+
=\frac{(-1)^{n+1}e^{-nf(z_+)}}{2^{n+2}\pi^2 }\lim_{\varepsilon\to0}\int_{-\frac{ \sqrt{-\log \varepsilon}}2}^{\frac{ \sqrt{-\log \varepsilon}}2}\,e^{-nu^2}\int_{\Gamma_0^+}\frac{g(z)}{w(z)-u}\,dz\,du,
\end{equation}
 where $\Gamma_0^+$ is the closed contour (oriented in the clockwise direction) defined in Notation \ref{No:Gam0+}, enclosing  $\Gamma^+_\varepsilon$ in its interior as depicted in Fig. \ref{steepest_decent2}.
To express $I_+$ as a series, we define the coefficient $\A_j(e^{i\gamma})$ 
 \begin{equation}\label{Eq:Aj2}
\begin{split}
\A_j(e^{i\gamma}):=&-\frac{\Gamma(j+1/2)\sqrt{(e^{2i\gamma}-1)}}{4\pi^{3/2}e^{i\gamma}(1-e^{i\gamma})^{\alpha}(e^{i\gamma}+1)^{\beta}}\int_{\Gamma_0^+}\,\frac{g(z)}{w(z)^{2j+1}}dz
\end{split}
\end{equation}
and the error term $\varepsilon_p^+(n,\alpha,\beta)$  
\begin{equation}\label{Eq:epsilon+}
\varepsilon_p^+(n,\alpha,\beta):=-\frac{\sqrt{(e^{2i\gamma}-1)n}}{4\pi^{3/2}e^{i\gamma}(1-e^{i\gamma})^{\alpha}(e^{i\gamma}+1)^{\beta}}
\lim_{\varepsilon\to0}\int_{\frac{ \sqrt{-\log \varepsilon}}2}^{-\frac{ \sqrt{-\log \varepsilon}}2}\,e^{-nu^2}u^{2p}C_{p,+}(u)du,
\end{equation}
with 
\begin{align}\label{Eq:Cpu2}
C_{p,+}(u)=\int_{\Gamma_0^+}\,\frac{g(z)}{w(z)^{2p}(w(z)-u)}dz,~~~
\end{align}
 for any ~$u\in (-\frac{ \sqrt{-\log \varepsilon}}2,\frac{ \sqrt{-\log \varepsilon}}2)$.
For any positive integer $p$, we have
\[\frac1{w(z)-u}=\sum_{j=1}^{2p}\frac{u^{j-1}}{w(z)^j}+\frac{u^{2p}}{w(z)^{2p}[w(z)-u]}\]
and
\begin{equation*}
\int_{\mathbb R}e^{-nu^2}u^{j-1}du=
\begin{cases}0, ~~~&j=2k\\
\frac{\Gamma(k+1/2)}{n^{k+1/2}},~~&j=2k+1.
\end{cases}
\end{equation*}
Consequently, from integrals \eqref{Eq:Ijc2}-\eqref{Eq:epsilon+}, we arrive at
\begin{equation}\label{Eq:I2juer}
\begin{split}
I_+=&\frac{(-1)^{n+1}e^{-nf(z_+)}}{2^{n+2}\pi^2}\lim_{\varepsilon\to0}\Big[\sum_{j=1}^{2p}\int_{\frac{ \sqrt{-\log \varepsilon}}2}^{-\frac{ \sqrt{-\log \varepsilon}}2}\,e^{-nu^2}u^{j-1}\,du\int_{\Gamma_0^+}\frac{g(z)}{w(z)^j}\,dz\\
&+\int_{\frac{ \sqrt{-\log \varepsilon}}2}^{-\frac{ \sqrt{-\log \varepsilon}}2}\,e^{-nu^2}u^{j-1}\int_{\Gamma_0^+}\frac{g(z)}{w(z)^{2p}[w(z)-u]}\,dz\,du\Big]\\
=&\frac{(-1)^ne^{i\gamma}(1-e^{i\gamma})^{\alpha}(e^{i\gamma}+1)^{\beta}e^{-nf(z_+)}}{2^{n} \sqrt{n\pi(e^{2i\gamma}-1)}}\Big\{\sum_{j=0}^{p-1}\frac{\A_j(e^{i\gamma})}{n^{j}}+\varepsilon_p^+(n,\alpha,\beta)\Big\}.
\end{split}
\end{equation}
Specifically, when $p=0$,  an application of the Cauchy integral formula yields
\[\A_0(e^{i\gamma})=-\frac{\Gamma(1/2)\sqrt{(e^{2i\gamma}-1)}}{4\pi^{3/2}e^{i\gamma}(1-e^{i\gamma})^{\alpha}(e^{i\gamma}+1)^{\beta}}\int_{\Gamma_0^+}\,\frac{g(z)}{[f(z)-f(z_+)]^{\frac{1}{2}}}dz=1.\]


To derive an upper bound for the remainder term $\varepsilon_p^+(n,\alpha,\beta)$, we first require an estimate for $C_{p,+}(u)$. 
This estimate relies on analyzing the complex function $w(z)$ on $\Gamma_{0}^+$.  Specifically, we must decompose $w(z)$ into its modules and imaginary components on $\Gamma_{0}^+$ to quantify their contributions to the bound for $C_{p,+}(u)$.
Since $\Gamma_{0}^+$ is the union of four arcs $\Gamma_{o}^+$, $\Gamma_{-1}^+$, $\Gamma_{o}^+$ and $\Gamma_{x}^+$ and four line segments, 
we need to split the discussion of $w(z)$ into three lemmas, one for the four arcs and two for four line segments  $\Gamma_{(o,1)}^+$, $\Gamma_{(1,x)}^+$, $\Gamma_{(x,-1)}^+$ and $\Gamma_{(-1,o)}^+$.

Now we discuss the error for $w(z)$ on the four arcs as follows. 

\begin{lem}\label{Le:wcir}
Let w(z) be defined as in equation \eqref{Eq:wz2} and $\Gamma_{-1}^+,\, \Gamma_{1}^+,\, \Gamma_{o}^+, \,\Gamma_{x}^+$ be the four arcs, with the parameter $\varepsilon$, defined in Notation \ref{No:Gam0+}. Then, the following bounds hold.
\begin{enumerate}
\item
   For $z\in \Gamma_{1}^+\cup \Gamma_{-1}^+$,  we have 
     \begin{align}\label{Eq:L|w|}
    |w(z)|  >1;
\end{align}
and for $u\in[-\frac{ \sqrt{-\log \varepsilon}}2, \frac{ \sqrt{-\log \varepsilon}}2]$
\begin{equation}\label{Eq:Gapm1w-u}
    |w(z)-u|>\frac14.
\end{equation}

\item For $z\in \Gamma_{o}^+\cup \Gamma_{x}^+$, we have 
\begin{align}\label{Eq:rew2}
|w(z)| >1;
\end{align}
and for any $u\in \R$
\begin{align}\label{Eq:|w-u|o}
 | w(z)-u|>1.
 \end{align}
\end{enumerate}

\end{lem}


In order to estimate the imaginary parts of $w(z)$ on $\Gamma_{(o,1)}^+$, $\Gamma_{(1,x)}^+$, $\Gamma_{(x,-1)}^+$ and $\Gamma_{(-1,o)}^+$,
 we need to split these segments into two parts, respectively.  
\begin{rem}\label{rem:Gapsli}
Let the four line segments $\Gamma_{(o,1)}^+$, $\Gamma_{(1,x)}^+$, $\Gamma_{(x,-1)}^+$ and $\Gamma_{(-1,o)}^+$ be defined as before, with $\gamma>0$ and $\varepsilon>0$  being a small parameter. We introduce the following decompositions.
\begin{enumerate}
\item Let $\Gamma_{(o,1)}^{+,l}:=[1+\varepsilon,  1+\sqrt{1-\cos\gamma}]^+$  and $\Gamma_{(o,1)}^{+,r}:=(1+\sqrt{1-\cos\gamma}, \cos\gamma+4e]^+$. We can write
    \[
\Gamma_{(o,1)}^+=\Gamma_{(o,1)}^{+,l}+\Gamma_{(o,1)}^{+,r};
    \] 
   \item    Let $ \Gamma_{(1,x)}^{+,l}:=[\cos\gamma+\frac{\sin^2\gamma}{8},\cos\frac{\gamma}{2} ]^+$ and $ \Gamma_{(1,x)}^{+,r}:=[\cos\frac{\gamma}{2}, 1-\varepsilon ]^+$. Then
    \[
    \Gamma_{(1,x)}^+= \Gamma_{(1,x)}^{+,l}+ \Gamma_{(1,x)}^{+,r};
    \]
    \item Let $\Gamma_{(x,-1)}^{+,l}:=[-1+\varepsilon, 1-\sqrt 3]^+$ and   $\Gamma_{(x,-1)}^{+,r}:=[1-\sqrt 3, \cos\gamma-\frac{\sin^2\gamma}{8}]^+$.
    We have
    \[
\Gamma_{(x,-1)}^+=\Gamma_{(x,-1)}^{+,l}+\Gamma_{(x,-1)}^{+,r};
    \]
   \item Let $\Gamma_{(-1,o)}^{+,l}:= [-4e+\cos\gamma, -1-\sqrt{\cos\gamma+1}]^+$ 
 and  $\Gamma_{(-1,o)}^{+,r}:=(-1-\sqrt{\cos\gamma+1}, -1-\varepsilon]^+$. Then
    \[
    \Gamma_{(-1,o)}^+=\Gamma_{(-1,o)}^{+,l}+\Gamma_{(-1,o)}^{+,r}.
    \]
   
    \end{enumerate}
\end{rem}
\begin{proof}
We only need to prove statements 2 and 3.

 2.  
When $0<\gamma<\frac\pi{2}$, it is easy to see from \eqref{Eq:12singa} that 
\begin{align*}
    \cos\gamma+1-2\cos\frac\gamma2+\frac{\sin^2\gamma}{8}&\leqslant 2\cos^2\frac \gamma 2-2\cos\frac\gamma2 +\frac{1}{2}\sin^2\frac \gamma 2
\leqslant \Big(\frac{1}{2}-\frac{3}{2}\cos\frac{\gamma}{2}\Big)(1-\cos\frac{\gamma}{2})<0.
\end{align*}
Thus, choosing $\varepsilon<1-\cos\frac{\gamma}{2}$, we have
\[
\cos\gamma+\frac{\sin^2\gamma}{8}<-1+2\cos\frac\gamma2<\cos\frac{\gamma}{2}<1-\varepsilon.
\]
Therefore, statement 2 holds.

3.  For $0<\gamma<\frac\pi{2}$, $\cos\gamma-\frac{\sin^2\gamma}{8}$ is decreasing.  
 Choosing $\varepsilon<2-\sqrt3$ gives
\begin{align*}
-1+\varepsilon< 1-\sqrt 3< \cos\gamma-\frac{\sin^2\gamma}{8}.
\end{align*}
Hence, statement 3 holds.

\end{proof}

Applying these notations, we have the following remark. 

 \begin{rem}\label{rem:log}
The following properties hold for the logarithmic functions. 
\begin{enumerate}
    \item 
The function $\log\frac{z^2-1}{2(z-\cos\gamma)}$   is strictly increasing along both  $\Gamma_{(o,1)}^+$ and $\Gamma_{(x,-1)}^+$. It has zeros exclusively at  $z=1+2\sin\frac{\gamma}{2}\in \Gamma_{(o,1)}^{+,r}$ and $z=1-2\sin\frac{\gamma}{2}\in\Gamma_{(x,-1)}^{+,r}$.

\item The function $\log\frac{z^2-1}{2(\cos\gamma-z)}$  is strictly decreasing along both $\Gamma_{(1,x)}^+$ and $\Gamma_{(-1,o)}^+$. It has zeros exclusively at  $z=-1+2\cos\frac{\gamma}{2}\in \Gamma_{(1,x)}^{+,l}$, and $z=-1-2\cos\frac{\gamma}{2}\in\Gamma_{(-1,o)}^{+,l}$.
\end{enumerate}
\end{rem}


 We now calculate the modulus and the imaginary parts of $w(z)$ over the line segments. 
 


 \begin{lem}\label{Lem:imw1}
 The following inequalities give lower bounds for the modulus and imaginary part of $w(z)$ across the specified regions defined as before. 
 \begin{enumerate}
     \item 
 For $z\in\Gamma_{(o,1)}^+\cup \Gamma_{(x,-1)}^+ \cup \Gamma_{(1,x)}^+\cup\Gamma_{(-1,o)}^+$, 
    the modulus of $w(z)$ satisfies
 \begin{align}\label{Eq:wz2o-1}
|w(z)|\geqslant \gamma^{\frac12}.
\end{align}


\item For $z\in\Gamma_{(o,1)}^{+,l}\cup \Gamma_{(1,x)}^{+,r}$, the imaginary part of $w(z)$ satisfies
\begin{equation}\label{Eq:imwz1o1}
|\Im w(z)|
\geqslant \frac{\gamma\sqrt{\alpha+1}}{580} \Big|\frac{z+1}{2(z-\cos\gamma)}\Big|^{\frac{\alpha+1}2} |z-1|^{\frac{\alpha+1}2}.
\end{equation}


\item For $z\in \Gamma_{(-1,o)}^{+,r}\cup \Gamma_{(x,-1)}^{+,l}$, the imaginary part of $w(z)$ satisfies
\begin{equation}\label{Eq:imwz2-1} 
|\Im w(z)|
\geqslant \frac{\gamma\sqrt{\beta+1}}{74 \pi(\pi-\gamma)} \Big|\frac{z-1}{2(z-\cos\gamma)}\Big|^{\frac{\beta+1}2} |z+1|^{\frac{\beta+1}2}.
\end{equation}


\end{enumerate}

 \end{lem}

 
\begin{lem}\label{Le:imw}
For  $z \in \Gamma_{(o,1)}^{+,r}\cup\Gamma_{(x,-1)}^{+,r}\cup  \Gamma_{(-1,o)}^{+,l}\cup \Gamma_{(1,x)}^{+,l}$, we have the following bound  for the imaginary part of $w(z)$
\begin{align}\label{Eq:imwz-1o2}
 |\Im w(z)|>\frac{\gamma}{2\sqrt{ 3\pi}}.
\end{align}
\end{lem}

For any $u\in (-\frac{ \sqrt{-\log \varepsilon}}2,\frac{ \sqrt{-\log \varepsilon}}2)$, in order to estimate the integrand $\frac{g(z)}{w(z)^{2p}(w(z)-u)}$ in \eqref{Eq:Cpu2}, we also need the following lemma. 

\begin{lem}\label{Le:g2}
Let $\Gamma_{-1}^+,\, \Gamma_{1}^+,\, \Gamma_{o}^+, \,\Gamma_{x}^+$ be the four arcs, as defined in Notation \ref{No:Gam0+}. Under the parameter constraints $-1<\beta\leqslant\alpha\leqslant 0$ and $\ell<0$, the following inequalities hold:
\begin{enumerate}
    \item 
 For $z\in  \Gamma_{-1}^+$, we have
\begin{equation}\label{Eq:gzGa-1-x}
    |z-\cos\gamma|\geqslant 1+\cos\gamma \qquad
\text{and} \qquad
    |z-1|^{\alpha}<1.
\end{equation}

\item For $z\in \Gamma_{1}^+$, 
we have
\begin{equation}\label{Eq:gzGa1-x}
    |z-\cos\gamma|\geqslant 1-\cos\gamma \qquad
\text{and} \qquad
    |z+1|^{\beta}<1.
\end{equation}

\item  For $z\in \Gamma_{o}^+$, we have
\begin{equation}\label{Eq:Gaoz-x1-1}
    |z-\cos\gamma|=4e,
\qquad
    |z-1|^{\alpha}<1
\qquad \text{and} \qquad 
    |z+1|^{\beta}<1.
\end{equation}

\item For $z\in \Gamma_{x}^+$, we have
\begin{equation}\label{Eq:Gaxz-x1-1}
    |z-\cos\gamma|=\frac{\sin^2\gamma}{8},
\qquad
    |z-1|^{\alpha}\leqslant 12 \gamma^{2\alpha}
\qquad \text{and} \qquad 
    |z+1|^{\beta}\leqslant 2.
\end{equation}

   \end{enumerate}
\end{lem}

Now we are in a position to give a lemma to estimate $C_{p,+}(u)$. 

For this purpose we first rewrite the integral \eqref{Eq:Cpu2} as follows
\begin{equation}\label{Eq:Cp2}
C_{p,+}(u)
=\int_{\Gamma^*}\frac{g(z)}{w(z)^{2p}(w(z)-u)}dz,
\end{equation}
where $\Gamma^*=\Gamma_{-1}^++\Gamma_{1}^++\Gamma_{o}^++\Gamma_x^++\Gamma_{(o,1)}^{+,r}+\Gamma_{(x,-1)}^{+,r}+\Gamma_{(1,x)}^{+,l}+\Gamma_{(-1,o)}^{+,l}+\Gamma_{(o,1)}^{+,l}+\Gamma_{(1,x)}^{+,r}+\Gamma_{(-1,o)}^{+,r}+\Gamma_{(x,-1)}^{+,l}$
for $u\in (-\frac{ \sqrt{-\log \varepsilon}}2, \frac{ \sqrt{-\log \varepsilon}}2)$;
and we present the following lemma  to estimate the error bound for  $C_{p,+}(u)$.

\begin{lem}\label{Lem:Cp+u}
Let $C_{p,+}(u)$ be the integral given in equation \eqref{Eq:Cp2}, and $\alpha$, $\beta$, $\gamma$ be the parameters as before. The modulus of $C_{p,+}(u)$ satisfies the following inequality for $u\in (-\frac{ \sqrt{-\log \varepsilon}}2, \frac{ \sqrt{-\log \varepsilon}}2)$ as $\varepsilon \to 0$
\begin{align*}
   \notag |C_{p,+}(u)|\leqslant &\pi +24 \pi \gamma^{2\alpha}+ \frac{128\sqrt{3\pi}(4e+1)^{\alpha+1}}{(\beta+1)\gamma^{p+3}}+\frac{128\sqrt{3\pi}}{\gamma^{p+3}}\mathbf{B}(\alpha+1,\beta+1) + \frac{9280 \gamma^{\alpha} + 1184\pi(\pi-\gamma)\gamma^{\beta}}{(\beta+1)^{\frac{3}{2}}\gamma^{p+1}}.
\end{align*}

\end{lem}

\begin{proof}

We split this proof into four steps. 

Step 1. First, we claim that 

\begin{align}\label{Eq:Ga-1}
&\lim_{\varepsilon\to0^+}\Big|\int_{\Gamma_{-1}^+}\frac{g(z)}{|w(z)|^{2p}|w(z)-u|}\,dz\Big|=0
\end{align}
and
\begin{align}\label{Eq:Ga1}
&\lim_{\varepsilon\to0^+}\Big|\int_{\Gamma_{1}^+}\frac{g(z)}{|w(z)|^{2p}|w(z)-u|}\,dz\Big|=0.
\end{align}
When $0<\varepsilon<\min\{\frac{1}{144},1-\cos\gamma\}$, it follows from inequalities \eqref{Eq:L|w|} and  \eqref{Eq:Gapm1w-u} in Lemma \ref{Le:wcir} that 
\begin{align*}
\lim_{\varepsilon\to0^+}\Big|\int_{\Gamma_{-1}^+}\frac{g(z)}{|w(z)|^{2p}|w(z)-u|}\,dz\Big|
\leqslant\lim_{\varepsilon\to0^+}\int_{\Gamma_{-1}^+}\frac{4|z-1|^{\alpha} |z+1|^\beta}{|z-\cos\gamma|}\,dz.
\end{align*}
Obviously, inequalities \eqref{Eq:gzGa-1-x} yields that
\begin{align*}
\lim_{\varepsilon\to0^+}\Big|\int_{\Gamma_{-1}^+}\frac{g(z)}{|w(z)|^{2p}|w(z)-u|}\,dz\Big|
\leqslant&  \frac{4\pi \varepsilon^{\beta+1}}{1+\cos\gamma}.
\end{align*}
Recalling $-1< \beta \leqslant 0$, equation \eqref{Eq:Ga-1} holds.

Applying inequalities \eqref{Eq:L|w|} and \eqref{Eq:Gapm1w-u} to $\lim_{\varepsilon\to0^+}\Big|\int_{\Gamma_{1}^+}\frac{g(z)}{|w(z)|^{2p}|w(z)-u|}\,dz\Big|$, the same discussion gives equation \eqref{Eq:Ga1}.

Step 2. For all $u\in (-\frac{ \sqrt{-\log \varepsilon}}2, \frac{ \sqrt{-\log \varepsilon}}2)$ with $\varepsilon \to 0$,  let the error of $C_{p,+}(u)$ in $\Gamma_{o}^+\cup \Gamma_{x}^+$ be denoted by 
\begin{equation*}
C_{p,+}^{o,x}(u)
=\int_{\Gamma_{o}^++ \Gamma_{x}^+}\frac{g(z)}{w(z)^{2p}(w(z)-u)}dz.
\end{equation*}
By inequalities \eqref{Eq:rew2} and \eqref{Eq:|w-u|o}, we have
\begin{align*}
|C_{p,+}^{o,x}(u) &|\leqslant \int_{\Gamma_{o}^+ +\Gamma_{x}^+}|g(z)|\,dz.
\end{align*}
 It follows from the expression of $g(z)$ in equation \eqref{Eq:gz2} that
\begin{align*}
|C_{p,+}^{o,x}(u) &|\leqslant \int_{\Gamma_{o}^+ +\Gamma_{x}^+}\frac{|z-1|^{\alpha}|z+1|^{\beta}}{|z-\cos\gamma|}dz.
\end{align*}
Given that $-1<\beta\leqslant\alpha \leqslant 0$,  inequalities \eqref{Eq:Gaoz-x1-1} and  \eqref{Eq:Gaxz-x1-1} yields  
\begin{align}\label{Eq:|Cpo|}
|C_{p,+}^{o,x}(u)| 
\leqslant \pi+24 \pi \gamma^{2\alpha}.
\end{align}

Step 3. For $u\in \R$, let the error of $C_{p,+}(u)$ in $\Gamma_{(o,1)}^{+,r}\cup \Gamma_{(x,-1)}^{+,r}\cup \Gamma_{(1,x)}^{+,l}\cup \Gamma_{(-1,o)}^{+,l}$ be denoted by
\begin{equation*}
C_{p,+}^{1}(u)
=\int_{\Gamma_{(o,1)}^{+,r}+\Gamma_{(x,-1)}^{+,r}+\Gamma_{(1,x)}^{+,l}+\Gamma_{(-1,o)}^{+,l}}\frac{g(z)}{w(z)^{2p}(w(z)-u)}dz.
\end{equation*}
Combined with \eqref{Eq:wz2o-1} and \eqref{Eq:imwz-1o2}, we have from $|w(z)-u|\geqslant |\Im w(z)|$ that
\begin{align}\label{Eq:Cp+1o}
\Big|C_{p,+}^{1}(u)\Big|\leqslant \frac{2\sqrt{3\pi}}{\gamma^{p+1}}\int_{\Gamma_{(o,1)}^{+,r}+\Gamma_{(x,-1)}^{+,r}+\Gamma_{(1,x)}^{+,l}+\Gamma_{(-1,o)}^{+,l}}|g(z)|\,dz. 
\end{align}
Recalling that $\Gamma_{(o,1)}^{+,r}=(1+\sqrt{1-\cos\gamma}, \cos\gamma+4e]^+$, $ \Gamma_{(1,x)}^{+,l}=[\cos\gamma+\frac{\sin^2\gamma}{8},\cos\frac{\gamma}{2} ]^+$, $\Gamma_{(x,-1)}^{+,r}=[1-\sqrt 3, \cos\gamma-\frac{\sin^2\gamma}{8}]^+$ and  $\Gamma_{(-1,o)}^{+,l}= [-4e+\cos\gamma, -1-\sqrt{\cos\gamma+1}]^+$, 
we obtain 
\begin{equation*}
    |z-\cos\gamma|\geqslant \frac{\sin^2\gamma}{8}\geqslant\frac{\gamma^2}{32},
\end{equation*}
for $z\in \Gamma_{(o,1)}^{+,r}\cup \Gamma_{(x,-1)}^{+,r}\cup \Gamma_{(1,x)}^{+,l}\cup \Gamma_{(-1,o)}^{+,l}$. Here we have used the inequality
\begin{align}\label{Eq:singa}
\sin\gamma\geqslant \frac{\gamma}{2},
\end{align}
for $\gamma\in(0,\frac{\pi}{2})$.
Therefore, combined with the expression of $g(z)=\frac{(1-z)^{\alpha} (z+1)^\beta}{z-\cos\gamma}$ (see equation \eqref{Eq:gz2}),  we can rewrite inequality \eqref{Eq:Cp+1o} as 
\begin{align*}
\Big|C_{p,+}^{1}(u)\Big|\leqslant \frac{64\sqrt{3\pi}}{\gamma^{p+3}}\int_{\Gamma_{(o,1)}^{+,r}+\Gamma_{(x,-1)}^{+,r}+\Gamma_{(1,x)}^{+,l}+\Gamma_{(-1,o)}^{+,l}}|1-z|^{\alpha} |z+1|^\beta\,dz. 
\end{align*}
Using the fact that 
\(
|z+1|> 1  
\)
for $\Gamma_{(o,1)}^{+,r}$ and 
\(
|1-z|>1
\) for $\Gamma_{(-1,o)}^{+,l}$, we have
\begin{align*}
 \Big|C_{p,+}^{1}(u)\Big|\leqslant&\frac{64\sqrt{3\pi}}{(\alpha+1)\gamma^{p+3}}(4e+\cos\gamma-1)^{\alpha+1}+\frac{64\sqrt{3\pi}}{(\beta+1)\gamma^{p+3}}(4e-\cos\gamma-1)^{\beta+1}\\
 &+\frac{64\sqrt{3\pi}}{\gamma^{p+3}}\int_{\Gamma_{(x,-1)}^{+,r}+\Gamma_{(1,x)}^{+,l}}(1-z)^{\alpha} (z+1)^\beta\,dz\\
 \leqslant&\frac{128\sqrt{3\pi} (4e+1)^{\alpha+1}}{(\beta+1)\gamma^{p+3}}+\frac{64\sqrt{3\pi}}{\gamma^{p+3}}\int_{-1}^1 (1-z)^{\alpha} (z+1)^\beta\,dz.
\end{align*}
Applying the Beta function
 \begin{align*}
    \int_{-1}^1 (1-z)^{\alpha} (z+1)^\beta\,dz =2^{\alpha+\beta+1}\int_{-1}^1 z^{\alpha} (1-z)^{\beta}\,dz =2^{\alpha+\beta+1}\mathbf{B}(\alpha+1,\beta+1),
\end{align*}
we have
\begin{align}\label{Eq:Cp+1c}
  \Big|C_{p,+}^{1}(u)\Big|
 \leqslant&\frac{128\sqrt{3\pi}(4e+1)^{\alpha+1}}{(\beta+1)\gamma^{p+3}}+\frac{128\sqrt{3\pi}}{\gamma^{p+3}}\mathbf{B}(\alpha+1,\beta+1).
\end{align}

Step 4.   For $u\in \R$, let the error of $C_{p,+}(u)$ in $\Gamma_{(o,1)}^{+,l}\cup \Gamma_{(1,x)}^{+,r}\cup\Gamma_{(-1,o)}^{+,r}\cup\Gamma_{(x,-1)}^{+,l}$ be denoted by
\begin{equation*}
C_{p,+}^{2}(u)
=\int_{\Gamma_{(o,1)}^{+,l}+\Gamma_{(1,x)}^{+,r}+\Gamma_{(-1,o)}^{+,r}+\Gamma_{(x,-1)}^{+,l}}\frac{g(z)}{w(z)^{2p}(w(z)-u)}dz.
\end{equation*}
The inequality \eqref{Eq:wz2o-1} gives
\begin{equation}\label{Eq:Cp+2o}
|C_{p,+}^{2}(u)|
\leqslant \frac{1}{\gamma^p}\left(\int_{\Gamma_{(o,1)}^{+,l}+\Gamma_{(1,x)}^{+,r}}\frac{|g(z)|}{|w(z)-u|}dz+\int_{\Gamma_{(-1,o)}^{+,r}+\Gamma_{(x,-1)}^{+,l}}\frac{|g(z)|}{|w(z)-u|}dz\right).
\end{equation}
Given that $|g(z)|=\frac{|z-1|^{\alpha} |z+1|^\beta}{|z-\cos\gamma|}$, the inequality \eqref{Eq:imwz1o1} yields 
\begin{align}\label{Eq:gzimw1}
\frac{|g(z)|}{|w(z)-u|}\leqslant\frac{580\cdot 2^{\frac{\alpha+1}{2}}}{\sqrt{\alpha+1}\gamma} |z+1|^{\beta-\frac{\alpha+1}{2}}|z-\cos\gamma|^{\frac{\alpha-1}{2}}|z-1|^{\frac{\alpha-1}2},
\end{align}
for $z\in \Gamma_{(o,1)}^{+,l}\cup\Gamma_{(1,x)}^{+,r}$; and the inequality \eqref{Eq:imwz2-1} yields 
\begin{align}\label{Eq:gzimw2}
\frac{|g(z)|}{|w(z)-u|}\leqslant\frac{74\cdot 2^{\frac{\beta+1}{2}}  \pi(\pi-\gamma)} {\sqrt{\beta+1}\gamma} |z-1|^{\alpha-\frac{\beta+1}{2}}|z-\cos\gamma|^{\frac{\beta-1}{2}}|z+1|^{\frac{\beta-1}2},
\end{align}
for $z\in \Gamma_{(-1,o)}^{+,r}\cup\Gamma_{(x,-1)}^{+,l}$.
Recalling that $\Gamma_{(o,1)}^{+,l}=[1+\varepsilon,  1+\sqrt{1-\cos\gamma}]^+$, $ \Gamma_{(1,x)}^{+,r}=[\cos\frac{\gamma}{2}, 1-\varepsilon ]^+$,  $\Gamma_{(-1,o)}^{+,r}=(-1-\sqrt{\cos\gamma+1}, -1-\varepsilon]^+$ and $\Gamma_{(x,-1)}^{+,l}=[-1+\varepsilon, 1-\sqrt 3]^+$, we have 
\begin{align*}
    |z+1|\geqslant 1 \qquad \text{and} \qquad |z-\cos\gamma|\geqslant \cos\frac{\gamma}{2}-\cos\gamma \geqslant \frac{\gamma^2}{4},
\end{align*}
for $\Gamma_{(o,1)}^{+,l}\cup\Gamma_{(1,x)}^{+,r}$;  similarly, we have
\begin{align*}
    |z-1|\geqslant 1 \qquad \text{and} \qquad |z-\cos\gamma|\geqslant\sqrt{3}-1\geqslant \cos\frac{\gamma}{2}-\cos\gamma \geqslant \frac{\gamma^2}{4},
\end{align*}
for $z\in \Gamma_{(-1,o)}^{+,r}\cup\Gamma_{(x,-1)}^{+,l}$ and $\gamma\in (0,\frac{\pi}{2}]$.
Therefore, for $-1<\beta\leqslant\alpha \leqslant 0$, inequalities \eqref{Eq:Cp+2o}-\eqref{Eq:gzimw2} yield that  
\begin{align*}
|C_{p,+}^{2}(u)|
\leqslant&\frac{580\cdot 2^{\frac{\alpha+1}{2}} \cdot 2^{1-\alpha}}{\sqrt{\alpha+1}\gamma^{p+1-\alpha}}\int_{\Gamma_{(o,1)}^{+,l}+\Gamma_{(1,x)}^{+,r}}|z-1|^{\frac{\alpha-1}2}\,dz+\frac{74\cdot 2^{\frac{\beta+1}{2}}  \pi(\pi-\gamma)\cdot 2^{1-\beta}} {\sqrt{\beta+1}\gamma^{p+1-\beta}}\int_{\Gamma_{(-1,o)}^{+,r}+\Gamma_{(x,-1)}^{+,l}}|z+1|^{\frac{\beta-1}2}\,dz.
\end{align*}
By a simple calculation, 
 this error term can be represented by
\begin{align}\label{Eq:Cp+2c}
\notag |C_{p,+}^{2}(u)|
\leqslant& \frac{4640 (1-\cos\gamma)^{\frac{\alpha+1}4}}{(\alpha+1)^{\frac{3}{2}}\gamma^{p+1-\alpha}}+\frac{4640 (1-\cos\frac{\gamma}{2})^{\frac{\alpha+1}2}}{(\alpha+1)^{\frac{3}{2}}\gamma^{p+1-\alpha}}\\
&+\frac{592 \pi(\pi-\gamma)(2-\sqrt{3})^{\frac{\beta+1}{2}}} {(\beta+1)^{\frac{3}{2}}\gamma^{p+1-\beta}}+\frac{148\pi(\pi-\gamma)\cdot 2^{\frac{3-\beta}{2}}(1+\cos\gamma)^{\frac{\beta+1}4}} {(\beta+1)^{\frac{3}{2}}\gamma^{p+1-\beta}}\\\notag
\leqslant& \frac{9280 \gamma^{\alpha} + 1184\pi(\pi-\gamma)\gamma^{\beta}}{(\beta+1)^{\frac{3}{2}}\gamma^{p+1}}.
\end{align}



Therefore, recalling the definitions of $C_{p,+}^{o,x}(u)$,   $C_{p,+}^{1}(u)$ and $C_{p,+}^{2}(u)$, this lemma holds with equations \eqref{Eq:Ga-1} and \eqref{Eq:Ga1}, inequalities \eqref{Eq:|Cpo|},  \eqref{Eq:Cp+1c} and \eqref{Eq:Cp+2c} combined.

\end{proof}

 \begin{lem}
  Maintain the notations in this section. The following inequality holds.
\begin{equation*}
|\varepsilon_p^+(n,\alpha,\beta)|\leqslant\frac{c_p^+\Gamma(p+\frac12)}{n^p},
\end{equation*}
where
\begin{align*}
c_p^+= & \pi +24 \pi \gamma^{2\alpha}+  \frac{128\sqrt{3\pi}(4e+1)^{\alpha+1}}{(\beta+1)\gamma^{p+3}}+\frac{128\sqrt{3\pi}}{\gamma^{p+3}}\mathbf{B}(\alpha+1,\beta+1) + \frac{9280 \gamma^{\alpha} + 1184\pi(\pi-\gamma)\gamma^{\beta}}{(\beta+1)^{\frac{3}{2}}\gamma^{p+1}}.
\end{align*}
 \end{lem}
 
\begin{proof}
It follows from equation \eqref{Eq:epsilon+} and Lemma \ref{Lem:Cp+u}  that
\begin{align*}
|\varepsilon_p^+(n,\alpha,\beta)|=
\left|-\frac{\sqrt{(e^{2i\gamma}-1)n}}{4\pi^{3/2}e^{i\gamma}(1-e^{i\gamma})^{\alpha}(e^{i\gamma}+1)^{\beta}}\lim_{\varepsilon\to0}\int_{-\frac{ \sqrt{-\log \varepsilon}}2}^{\frac{ \sqrt{-\log \varepsilon}}2}\,e^{-nu^2}u^{2p}C_{p,+}(u)du\right|\leqslant\frac{c_p^+\Gamma(p+\frac12)}{n^p}.
\end{align*}
To reach the last inequality, we have made use of the formulas
\[
\sqrt{e^{2i\gamma}-1}\leqslant \sqrt{2}, \qquad |e^{i\gamma}-1|^{-\alpha}\leqslant \sqrt{2} , \qquad |e^{i\gamma}+1|^{-\beta}\leqslant \sqrt{2}.
\]
\end{proof}

 By symmetry,  $|\varepsilon_p^-(n,\alpha,\beta)|$ has the same estimate. 

 We define 
 \begin{align*}
 \varepsilon_p(n,\alpha,\beta):=\Re (\varepsilon_p^+(n,\alpha,\beta)e^{(N\gamma+\kappa)i}),
 \end{align*}
 where $N=\frac{\alpha+\beta+1}{2}+n$ and $\kappa=-\Big(\frac{\alpha}{2}+\frac{1}{4}\Big)\pi$.
It follows from equation \eqref{Eq:I2juer} that 
 \begin{align*}
 \varepsilon_p(n,\alpha,\beta)&\leqslant\left|\Re \big(\A_p(e^{i\gamma})e^{(N\gamma+\kappa)i}\big)\right|\frac1{n^p}+|\Re (\varepsilon_{p+1}^+(n,\alpha,\beta)|\\
 &\leqslant\left|\Re \big(\A_p(e^{i\gamma})e^{(N\gamma+\kappa)i}\big)\right|\frac1{n^p}+|\varepsilon_{p+1}^+(n,\alpha,\beta)|.
 \end{align*}
It follows from the fact $I_-=\overline{I_+}$, equation \eqref{Eq:Jc2} can be rewritten as 
\begin{align}\label{Eq:Jc3}
    P_n^{(\alpha,\beta)}(\cos\gamma)=
    2(1-\cos\gamma)^{-\alpha}(\cos\gamma+1)^{-\beta} \Re I_+.
\end{align}
Recall the formula
\[
f(z_\pm)=-\log2\pm i(\pi -\gamma),
\]
and note that
\begin{align*}
\frac{(-1)^ne^{i\gamma}(1-e^{i\gamma})^{\alpha}(e^{i\gamma}+1)^{\beta}e^{-nf(z_+)}}{2^n (1-\cos\gamma)^{\alpha}(1+\cos\gamma)^{\beta}\sqrt{e^{2i\gamma}-1}}
=\frac{e^{(N\gamma+\kappa)i} }{2\sin^{\alpha+\frac{1}{2}}\frac{\gamma}{2} \cos^{\beta+\frac{1}{2}}\frac{\gamma}{2}}.
\end{align*}
By equation \eqref{Eq:I2juer}, the Jacobi polynomials in equation \eqref{Eq:Jc3} can be represented by 
\begin{equation*}
    P_n^{(\alpha,\beta)}(\cos\gamma)
=\frac{1}{\sin^{\alpha+\frac{1}{2}}\frac{\gamma}{2} \cos^{\beta+\frac{1}{2}}\frac{\gamma}{2}\sqrt {n\pi}}\Big\{\sum_{j=0}^{p-1}\frac{\Re\Big(\A_j(e^{i\gamma})e^{(N\gamma+\kappa)i}\Big)}{n^{j}}+\varepsilon_p(n,\alpha,\beta)\Big\}.
\end{equation*}

Now we can present the main theorem of this case, as stated in Theorem \ref{Th:C2}.

\section{The computation for the coefficients}\label{APP:A}

We start from equation \eqref{Eq:Aj1}. By applying Cauchy's theorem, we can deform the contour of the integration to a small loop around $z_+=e^{\gamma}$ with counter-clockwise orientation. By setting  $u=z e^{-\gamma}$, equation \eqref{Eq:Aj1} can be rewritten as 
\begin{equation}\label{Eq:aju}
\A_j(e^{\gamma})=-\frac{\Gamma(j+1/2)\sqrt{e^{2\gamma}-1}}{4\pi^{3/2}(e^{\gamma}-1)^{\alpha}(e^{\gamma}+1)^{\beta}}\oint_{(1+)}\,\frac{g(ue^{\gamma})}{(w(ue^{\gamma})^2)^{j+1/2}}du.
\end{equation}
To compute these coefficients, we require the following power expansions around $u=1$:
\begin{align}\label{Eq:Az-1}
\begin{split}
(z-1)^\alpha&=e^{\gamma\alpha}(u-1+(1-e^{-\gamma}))^{\alpha}=(e^{\gamma}-1)^\alpha\sum_{k=0}^{\infty}\frac{(-1)^k(-\alpha)_k}{k!}\Big(\frac{u-1}{1-e^{-\gamma}}\Big)^k,
\end{split}
\end{align}
\begin{align}\label{Eq:Az+1}
\begin{split}
(z+1)^\beta&=e^{\gamma\beta}(u-1+(1+e^{-\gamma}))^\beta=(e^{\gamma}+1)^\beta\sum_{k=0}^{\infty}\frac{(-1)^k(-\beta)_k}{k!}\Big(\frac{u-1}{1+e^{-\gamma}}\Big)^k
\end{split}
\end{align}
and
\begin{align}\label{Eq:Az-x}
\begin{split}
\frac1{z-\cosh\gamma}&=\frac1{e^\gamma}\frac{1}{u-1+e^{-\gamma}\sinh\gamma}=\frac{2}{e^{\gamma}-e^{-\gamma}}\sum_{k=0}^\infty(-1)^k\Big(\frac{2(u-1)}{1-e^{-2\gamma}}\Big)^k.
\end{split}
\end{align}
By a straightforward calculation, it will be shown
\[
\sum_{j=0}^\infty \alpha_jz^j\sum_{j=0}^\infty \beta_jz^j\sum_{j=0}^\infty \gamma_jz^j=\sum_{k=0}^\infty \Big(\sum_{j=0}^{k-i}\sum_{i=0}^k \alpha_i\beta_j\gamma_{k-i-j}\Big)z^k.
\]
 Thus, from equations \eqref{Eq:gz1} and \eqref{Eq:Az-1}-\eqref{Eq:Az-x}, the power series expansion of $g(ue^{\gamma})$ 
 around $u=1$ can be expressed as
\begin{align}\label{Eq:Ag1}
\notag g(ue^{\gamma})&=\frac{2(e^{\gamma}-1)^{\alpha}(e^{\gamma}+1)^{\beta}}{e^{\gamma}-e^{-\gamma}}\sum_{k=0}^\infty \Big(\sum_{j=0}^{k-i}\sum_{i=0}^k \frac{(-1)^{k}(-\alpha)_i(-\beta)_j\cdot 2^{k-i-j}} {i!j!(1-e^{-\gamma})^i(1+e^{-\gamma})^j(1-e^{-2\gamma})^{k-i-j}}\Big)(u-1)^k\\
&=\frac{2(e^{\gamma}-1)^{\alpha}(e^{\gamma}+1)^{\beta}}{e^{\gamma}-e^{-\gamma}}h(ue^{\gamma}),
\end{align}
where
\[
h(ue^{\gamma})=\sum_{k=0}^\infty (-1)^{k}a_k(e^{\gamma})(u-1)^k
\]
with
\begin{align*}
a_k(t)=\sum_{j=0}^{k-i}\sum_{i=0}^k \frac{2^{k-i-j}(-\alpha)_i(-\beta)_jt^{2k-i-j}}{i!j!(t-1)^i(t+1)^j(t^2-1)^{k-i-j}}.
\end{align*}
The first three terms for $t=e^{\gamma}$ are
\[
a_0(e^{\gamma})=1, \qquad a_1(e^{\gamma})=\frac{(2-\alpha-\beta)e^{2\gamma}-(\alpha-\beta)}{e^{2\gamma}-1},
\]
\[
 a_2(e^{\gamma})=\frac{4e^{4\gamma}}{(e^{2\gamma}-1)^2}-\Big(\frac{2\alpha e^{3\gamma}}{(e^{\gamma}-1)(e^{2\gamma}-1)}+\frac{2\beta e^{3\gamma}}{(e^{\gamma}+1)(e^{2\gamma}-1)}\Big)+\Big(\frac{\alpha(\alpha-1) e^{2\gamma}}{2(e^{\gamma}-1)^2}+\frac{\beta(\beta-1) e^{2\gamma}}{2(e^{\gamma}+1)^2}\Big)+\frac{\alpha\beta e^{2\gamma}}{e^{2\gamma}-1}.
\]
Now the coefficient $\A_j(e^{\gamma})$ in equation \eqref{Eq:aju} can be represented by
\begin{equation}\label{Eq:ajuh}
\A_j(e^{\gamma})=-\frac{\Gamma(j+1/2)e^{\gamma}}{2\pi^{3/2}\sqrt{e^{2\gamma}-1}}\oint_{(1+)}\,\frac{h(ue^{\gamma})}{(w(ue^{\gamma})^2)^{j+1/2}}du.
\end{equation}
It follows from equation \eqref{Eq:wz1} that  we have the following power series around $1$:
\begin{equation*}
\begin{split}
w(ue^{\gamma})^2&=\log(ue^{\gamma}-\cosh\gamma)-\log(ue^{\gamma}+1)-\log(ue^{\gamma}-1)+\log (2e^{ \gamma})\\
&=-\frac{e^{2\gamma}}{e^{2
\gamma}-1}(u-1)^2\Big(1+\frac{e^{2
\gamma}-1}{e^{2\gamma}}\sum_{k=1}^\infty (-1)^{k} b_{k+2}(e^{\gamma})(u-1)^k\Big),
\end{split}
\end{equation*}
where
\begin{align*}
b_{k+2}(t)=\frac{t^{k+2}}{k+2}\frac{2^{k+2}t^{k+2}-\left((t-1)^{k+2}+(t+1)^{k+2}\right)}{(t^2-1)^{k+2}},
\end{align*}
for $k\geqslant1.$ The first two terms for $t=e^{\gamma}$ are 
\[
b_3(e^{\gamma})=\frac{2e^{4\gamma}}{(e^{2\gamma}-1)^2}, \qquad b_4(e^{\gamma})=\frac{e^{4\gamma}(7e^{2\gamma}+1)}{(e^{2\gamma}-1)^3}.
\]

Applying Lauwerier's method \cite[Equs. (8) and (9)]{L51} (with a slightly different notation $Q_j(s)=(-1)^j q_j(-s)$), equation \eqref{Eq:ajuh} becomes 
\begin{align}\label{Eq:Ajga}
\A_{j}(e^{\gamma})&=\frac{ (-1)^{j}}{\sqrt{\pi}}\Big(\frac{e^{2
\gamma}-1}{e^{2\gamma}}\Big)^j\int_0^{\infty}e^{-s}s^{j-1/2}Q_{2j}\Big(\frac{e^{2
\gamma}-1}{e^{2\gamma}}s\Big)ds,
\end{align}
where $Q_0(\tau)=1$, and 
\begin{align*}
Q_j(\tau)=a_j(e^{\gamma})-\sum_{k=1}^j b_{k+2}(e^{\gamma})\int_0^{\tau} Q_{j-k}(s)ds.
\end{align*}
The first two terms are given by
\[
Q_1(\tau)=a_1(e^{\gamma})-b_3(e^{\gamma})\tau,
\]
\[
Q_2(\tau)=a_2(e^{\gamma})-\big(a_1(e^{\gamma})b_3(e^{\gamma})+b_4(e^{\gamma})\big)\tau+\frac{b_3(e^{\gamma})^2}{2}\tau^2.
\]
A simple calculation yields  
\[
\A_{0}(e^{\gamma})=1.
\]
and
\begin{align*}
\A_{1}(e^{\gamma})&=-\frac{15b_3^2}{16}\Big(\frac{e^{2
\gamma}-1}{e^{2\gamma}}\Big)^3+\frac{3(a_1b_3+b_4)}{4}\Big(\frac{e^{2
\gamma}-1}{e^{2\gamma}}\Big)^2-\frac{a_2}{2}\frac{e^{2
\gamma}-1}{e^{2\gamma}}.
\end{align*}
Performing the substitution $t=e^{\gamma}$, \eqref{Eq:Ajga} gives the more general form 
\begin{align*}
\A_{j}(t)&=\frac{ (-1)^{j}}{\sqrt{\pi}}\Big(\frac{t^2-1}{t^2}\Big)^j\int_0^{\infty}e^{-s}s^{j-1/2}Q_{2j}\Big(\frac{t^2-1}{t^2}s\Big)ds
\end{align*}
where $Q_0(\tau)=1$, and 
\begin{align*}
Q_j(\tau)=a_j(t)-\sum_{k=1}^j b_{k+2}(t)\int_0^{\tau} Q_{j-k}(s)ds.
\end{align*}

From equation \eqref{Eq:Aj2},   the contour of integration can be changed to a small loop around $z_+=e^{i \gamma}$ with counter-clockwise orientation. Introducing the substitution  $u=ze^{-i\gamma}$ yields 
\begin{equation}\label{Eq:Ajga2}
\frac{\Gamma(j+1/2)\sqrt{e^{2i\gamma}-1}}{4\pi^{3/2}(1-e^{i\gamma})^{\alpha}(e^{i\gamma}+1)^{\beta}}\oint_{(1+)}\,\frac{g(ue^{i\gamma})}{(w(ue^{i\gamma})^2)^{j+1/2}}du,
\end{equation}
where 
\begin{align}\label{Eq:Ag2}
g(ue^{i\gamma})=\frac{2(1-e^{i\gamma})^{\alpha}(e^{i\gamma}+1)^{\beta}}{e^{i\gamma}-e^{-i\gamma}}\sum_{k=0}^\infty a_k(e^{i\gamma})(-1)^{k}(u-1)^k
\end{align}
is slightly different from \eqref{Eq:Ag1}. 
Since $\cos\gamma=\cosh(i\gamma)$, the equations \eqref{Eq:Ag1} and \eqref{Eq:Ag2} confirm that the integral \eqref{Eq:Ajga2} is identical to the integral \eqref{Eq:aju} when $i\gamma$ is replaced by $\gamma$. This confirms that the right-hand side of \eqref{Eq:Aj2} can be represented by $\mathcal{A}(e^{i\gamma})$. 

\begin{appendices}
\section{Proof of part of  the lemmas}\label{APP:L}

\subsection{Proof of Lemma \ref{Le:|w||w-u|1}}\label{APP:L1}
Case 1. If $z\in \Gamma_o^0$, we can express $z$ in the form $z=\cosh\gamma+r e^{i\theta}$, where the radial parameter $r$ is subject to $r=10(\cosh\gamma+1)>\max\{\cosh\gamma+1,e^{\gamma}\}$, and $\theta \in(-\pi, \pi)$. 
With this parametrization in place, we can now apply the result from equation \eqref{Eq:wz1} to deduce that
\begin{equation*}
\begin{split}
|w(z)|^2=&\Big|\log re^{i\theta}-\log(z^2-1)+\log\big(2e^{\gamma}\big)\Big|
\geqslant|\log((r-\cosh\gamma)^2-1)|-|\log r|-\log\big(2e^{\gamma}\big)\\
\geqslant& \log[4(\cosh\gamma+1)]-\gamma\geqslant 4\delta^2.
\end{split}
\end{equation*}
Here we have used the definition of $\delta$ in equation \eqref{Eq:delta} to deduce the last inequality.  This implies that
 \begin{align}\label{Eq:wo}
 | w(z)|
 \geqslant 2\delta .
 \end{align}

Case 2. When $z\in \Gamma_{(x,o)^{\pm}}$, we can exploit the inherent symmetry of the problem to simplify our analysis. Without loss of generality, we focus on the subset $z\in \Gamma_{(x,o)^+}$.  For such $z$, we express $w(z)=\mu+i\nu$.  To thoroughly examine the behavior of w(z) within this region, we further divide our analysis into three distinct subcases:

Subcase 2.1. When $z=a\in \Gamma_{(x,o)^+}^1=[-10(\cosh\gamma+1)+\cosh\gamma, -1]^+ $,   we select the branch cut of the complex logarithm defined by equation \eqref{Eq:wz1} so that
\begin{equation*}
(\mu+i\nu)^2=-\log\frac{a^2-1}{x-a}-\pi i+\log (2e^{\gamma}).
\end{equation*}
By equating the imaginary parts on both sides of this equation, we find
\[
\mu =-\frac{\pi}{2\nu},
\]
which leads to the inequality
 \begin{align}\label{Eq:|w|NG+pi}
|w^2(z)|=\mu^2+\nu^2\geqslant \pi\geqslant 4\delta^2.
\end{align}

Subcase 2.2. For $z\in 
   \Gamma_{(x,o)^+}^2=  [-1, 1]^+$, the function $w^2(z)$ 
 is real-valued and exhibits a specific trend: it decreases from $+\infty$ to $w^2(z_-)$ as $z$ moves from $-1$ to $z_-$, and then increases from $w^2(z_-)$ back to $+\infty$  as $z$ approaches $1$. This behavior is visually depicted in Fig. \ref{Fig:w1}.  
Consequently, we deduce from equation \eqref{Eq:wz1} that
\begin{align*}
|w(z)|^2\geqslant  w^2(z_-)=  2\gamma\geqslant 4\delta^2.
\end{align*}

Subcase 2.3. When $z\in\Gamma_{(x,o)^+}^3=[1, \cosh\gamma-\frac{\cosh\gamma-1}{4(\cosh\gamma+2)^4}]^+$, by employing a similar argument to that used in Subcase 2.1, we find
\begin{align*}
\mu =\frac{\pi}{2\nu},
\end{align*}
and \eqref{Eq:|w|NG+pi} holds.

By a combination of the results of all three subcases, we conclude that
   \begin{align}\label{Eq:|w|-1}
 | w(z)|\geqslant 2\delta.
 \end{align}

Case 3. For $z\in\Gamma_x^0$, express $z$ as $z=\cosh\gamma+re^{i\theta}$, where $\theta \in \big(-\pi,\pi\big)$ and $r=\frac{\cosh\gamma-1}{4(\cosh\gamma+2)^4}$. Applying equation \eqref{Eq:wz1}, we obtain the following inequalities:
\begin{equation*}
\begin{split}
|w(z)|^2=&\Big|\log re^{i\theta}-\log(z^2-1)+\log(2e^{\gamma})\Big|
\geqslant|\log r|-|\log((\cosh\gamma+r)^2+1)|-\log(2e^{\gamma})\\
\geqslant&\log\frac{(\cosh\gamma+2)^2}{\cosh\gamma(\cosh\gamma-1)}\geqslant 4\delta^2.
\end{split}
\end{equation*}
This establishes the lower bound
 \begin{align}\label{Eq:|wz|x}
 | w(z)|
 \geqslant 2\delta.
 \end{align}

For $u\in[-\delta,\delta]$ and $z\in \Gamma^0$, inequality \eqref{Eq:le|w|} is obtained directly from \eqref{Eq:wo}, \eqref{Eq:|w|-1}, and \eqref{Eq:|wz|x}. Inequality \eqref{Eq:le|w-u|} is derived similarly, but also utilizes the triangle inequality.

\subsection{Proof of Lemma \ref{Le:g1}}\label{APP:L2}

We proceed to prove each statement of this lemma.

  1.  When $z\in\Gamma_x^0$, we express $z=\cosh\gamma+r e^{i\theta}$, where $r=\frac{\cosh\gamma-1}{4(\cosh\gamma+2)^4}<\min\{\frac{\cosh\gamma-1}{4},\frac{1}{4}\}$ and $\theta \in(-\pi, \pi)$. It then follows that
\[0<\cosh\gamma-\frac{\cosh\gamma-1}{4(\cosh\gamma+2)^4}-1\leqslant |z-1|\leqslant \cosh\gamma+\frac{\cosh\gamma-1}{4(\cosh\gamma+2)^4}-1,\]
and
\[\cosh\gamma-\frac{\cosh\gamma-1}{4(\cosh\gamma+2)^4}+1\leqslant |z+1|\leqslant \cosh\gamma+\frac{\cosh\gamma-1}{4(\cosh\gamma+2)^4}+1.\]
Given  $-1<\alpha\leqslant 0$ and $\beta\leqslant \alpha$, we can derive the following bounds
 \begin{align*}
&\max \Big\{\Big(\cosh\gamma-\frac{\cosh\gamma-1}{4(\cosh\gamma+2)^4}-1\Big)^{\alpha}, \Big(\cosh\gamma+\frac{\cosh\gamma-1}{4(\cosh\gamma+2)^4}-1\Big)^{\alpha}\Big\}\\
&\leqslant \Big(\cosh\gamma-\frac{\cosh\gamma-1}{4(\cosh\gamma+2)^4}-1\Big)^{\alpha}\leqslant 2(\cosh\gamma-1)^{\alpha}
\end{align*}
and
 \begin{align*}
 &\max \Big\{\Big(\cosh\gamma-\frac{\cosh\gamma-1}{4(\cosh\gamma+2)^4}+1\Big)^{\beta},\Big(\cosh\gamma+\frac{\cosh\gamma-1}{4(\cosh\gamma+2)^4}+1\Big)^{\beta}\Big\}\\
&\leqslant \Big(\cosh\gamma-\frac{\cosh\gamma-1}{4(\cosh\gamma+2)^4}+1\Big)^{\beta}\leqslant 1.
\end{align*}
Substituting these bounds into the expression for $g(z)$ given in equation \eqref{Eq:gz1} yields inequality \eqref{Eq:Gax0g}.

2.  For $z\in \Gamma_{(x,o)^+}^1$, we have $|z-\cosh\gamma|\geqslant \cosh\gamma+1$ and  $2\leqslant |z-1|\leqslant 10(\cosh\gamma+1)-\cosh\gamma+1$; inequality \eqref{Eq:Gaxo1g} is directly obtained from the expression of $g(z)$ in equation \eqref{Eq:gz1}.

For $z\in \Gamma_{(x,o)^+}^2$, noting that $|z-\cosh\gamma|\geqslant \cosh\gamma-1$, 
inequality \eqref{Eq:Gaxo2g} follows from the expression of $g(z)$ in equation \eqref{Eq:gz1}.

For $z\in \Gamma_{(x,o)^+}^3$, we have $|z-\cosh\gamma|\geqslant \frac{\cosh\gamma-1}{4(\cosh\gamma+2)^4}$ and $2 \leqslant |z+1|\leqslant  \cosh\gamma+1-\frac{\cosh\gamma-1}{4(\cosh\gamma+2)^4}$. Since $\beta\leqslant \alpha\leqslant0$,  inequality \eqref{Eq:Gaxo3g} holds because $\max \{2^{\beta}, \big( \cosh\gamma+1-\frac{\cosh\gamma}{4(\cosh\gamma+2)^4}\big)^{\beta}\}\leqslant 1$.

3. For $z\in\Gamma_{o}^0$, we have $|z-\cosh\gamma|=10(\cosh\gamma+1)$, $10(\cosh\gamma+1)-\cosh\gamma+1\leqslant |z-1|\leqslant 10(\cosh\gamma+1)+\cosh\gamma-1$  and $ 10(\cosh\gamma+1)-\cosh\gamma-1\leqslant |z+1|\leqslant 10(\cosh\gamma+1)+\cosh\gamma+1$ . Given  $-1<\alpha\leqslant 0$ and $\beta\leqslant \alpha$, a simple computation yields 
 \begin{align*}
\max \Big\{\Big(9\cosh\gamma+11\Big)^{\alpha}, \Big(11\cosh\gamma+9\Big)^{\alpha}\Big\}\leqslant 1
\end{align*}
and 
 \begin{align*}
\max \Big\{\Big(9(\cosh\gamma+1)\Big)^{\beta},\Big(11(\cosh\gamma+1)\Big)^{\beta}\Big\}\leqslant 1.
\end{align*}
Thus, inequality \eqref{Eq:Gao0g} is established using the expression of $g(z)$  in equation \eqref{Eq:gz1}.

4. For $z\in\Gamma_{11}$, recalling the definition of $\Gamma_{11}$ in \eqref{Eq:Gam11}, we get $|z-\cosh\gamma|=\sinh\gamma$. Given $\gamma>0$, $-1<\alpha\leqslant 0$ and $\beta\leqslant \alpha$, we derive
 \begin{align*}
&\max|z-1|^{\alpha}\leqslant\max \Big\{(1-z_-)^{\alpha}, (z_+-1)^{\alpha}\Big\}\leqslant\max \Big\{(1-e^{-\gamma})^{\alpha}, (e^{\gamma}-1)^{\alpha}\Big\}\leqslant \Big(\frac{e^{\gamma}-1}{e^{\gamma}}\Big)^{\alpha},\\
&\max|z+1|^{\beta}\leqslant\max \Big\{
(e^{-\gamma}+1)^{\beta},(e^{\gamma}+1)^{\beta}\Big\}\leqslant 1.
 \end{align*}
Applying the expression of $g(z)$  in equation \eqref{Eq:gz1} gives \eqref{Eq:Ga11g}.
 
For $z\in[e^{-\gamma},1]$, we have $|z-\cosh\gamma|=\cosh\gamma-1$ and 
 \begin{align*}
\max|z+1|^{\beta}\leqslant\max \Big\{
(e^{-\gamma}+1)^{\beta},2^{\beta}\Big\}\leqslant 1.
 \end{align*}
 Therefore, inequality \eqref{Eq:e-ga1g} holds.

\subsection{Proof of Lemma \ref{Lem:zz-x}}\label{APP:L5}

For any $z\in\Gamma^+_{\varepsilon}$,  we first derive the modulus bound. By construction of $ \Gamma^+_{\varepsilon}$ via equation \eqref{Eq:zs}, we have
\[
|z|\leqslant |re^{i\gamma}|+ \sqrt{(1-r)|1-re^{2\gamma i}|}=1+\sqrt 2,
\]
 where $0\leqslant r\leqslant 1$.
 This establishes the first inequality \eqref{Eq:|z|le1}.

Next, we address the lower bound for $|z-\cos\gamma|$.
From the parametrization \eqref{Eq:zs}, we express $z-\cos \gamma$ as
\begin{align*}
z-\cos \gamma=s-\cos \gamma\pm \sqrt{(s-\cos \gamma)^2+\sin^2\gamma},
\end{align*}
where $s=r e^{i \gamma}$.
Applying the triangle inequality to $ |z - \cos\gamma| $, we obtain
\begin{align*}
|z-\cos\gamma|&\geqslant \sqrt{\sin^2\gamma+|s-\cos \gamma|^2}-|s-\cos \gamma|.
\end{align*}
To simplify further, we use the inequality \eqref{Eq:s-x}, which guarantees $ |s - \cos \gamma| \leqslant 1 $. Substituting this bound, we find
\begin{align*}
|z-\cos\gamma|
\geqslant  \sqrt{\Big(\frac{\sin^2\gamma}{3}+|s-\cos \gamma|\Big)^2}-|s-\cos \gamma|.
\end{align*}
Thus, inequality \eqref{Eq:|z-x|} holds uniformly for all $ z \in \Gamma^+_{\varepsilon} $.

\subsection{Proof of Lemma \ref{Le:wcir} }\label{APP:L6}

1. 
For $z\in \Gamma_{1}^+$, by applying equation \eqref{Eq:wz2} and the triangle inequality, we find
    \begin{align*}
    |w(z)|=&\Big|\Big[\log(z-x)-\log(1-z)-\log(1+z)+\log2-(\pi -\gamma)i\Big]^{1/2}\Big| > \sqrt{\frac{-\log \varepsilon}2}.
    \end{align*}
Since inequality \eqref{Eq:var2} implies that $0<\varepsilon<\min\{\frac{1}{144},1-\cos\gamma\}$, we have
\[
|w(z)|\geqslant \sqrt{\frac{|\log\varepsilon|}2}=\sqrt{\log 12}>1.
\]
For $u\in[-\frac{ \sqrt{-\log \varepsilon}}2, \frac{ \sqrt{-\log \varepsilon}}2]$, the reverse triangle inequality gives
\[
|w(z)-u|\geqslant |w(z)| -|u| \geqslant \frac{ (\sqrt{2}-1)\sqrt{-\log \varepsilon}}2>\frac14.
\]
    
The same argument applies to $z \in \Gamma_{-1}^+$, and one
obtains inequality \eqref{Eq:L|w|} for $z\in \Gamma_{-1}^+$.

2.  
Set $z=\cos\gamma+r e^{i\theta}$.  It follows from equation \eqref{Eq:wz2} that
\begin{equation*}
\begin{split}
w(z)=&\Big[-\log\big(-re^{i\theta}-2\cos\gamma-\frac{(\cos\gamma^2-1)e^{-i\theta}}{r}\big)+\log2-(\pi-\gamma)i\Big]^{1/2}
\\=&\Big[-\frac12\log\big(r^2+\frac{\sin^4\gamma}{r^2}-2\sin^2\gamma\cos(2\theta)+4\cos\gamma(r-\frac{\sin^2\gamma}{r})\cos\theta+4\cos^2\gamma\big)
\\
&+\log2+i(\gamma-\arctan\frac{\sin\theta(r^2+\sin^2\gamma)}{\cos\theta(r^2-\sin^2\gamma)+2r \cos\gamma})\Big]^{1/2}.
\end{split}
\end{equation*}
This implies 
\begin{align}\label{Eq:phio}
 \Re w^2(z)&=-\frac12\log\Big(\frac{r^2}4+\frac{\sin^4\gamma}{4r^2}-\frac{\sin^2\gamma\cos(2\theta)}2+\cos\gamma(r-\frac{\sin^2\gamma}{r})\cos\theta+\cos^2\gamma\Big).
\end{align}
For $r > 1$ and $\theta \in (0, \pi)$, we observe from equation \eqref{Eq:phio} that 
\[
 \Re w^2(z)\leqslant -\frac12\log\Big(\frac{r^2}4-\frac12-r\Big).
\]
When $z\in \Gamma_{o}^+$, choosing $r = 4e$, we obtain
\begin{equation}\label{Eq:rew2z}
\Re w^2(z)\leqslant -\frac 12\log (4e(e-1)-\frac 12 )<-1. ~~~  
\end{equation}
For $r<\sin^2\gamma ~{\rm and}~\theta\in[0,\pi]$, equation \eqref{Eq:phio} yields 
\begin{align*}
 \Re w^2(z)\leqslant -\frac12\log\Big( \frac{r^2}4+\frac{\sin^4\gamma}{4r^2}-\frac12-\frac{\sin^2\gamma}{r}+r+\cos^2\gamma\Big).
\end{align*}
When $z\in \Gamma_{x}^+$, choosing $r=\frac{\sin^2\gamma}{8}$, we have 
\begin{equation}\label{Eq:rew2x}
\Re w^2(z)\leqslant -\frac12\log(\frac{\sin^4\gamma}{256}+\frac{15}2+\frac{\sin^2\gamma}{8}+\cos^2\gamma)< -1. 
\end{equation}
Since $|w(z)|\geqslant |\Re w^2(z)|^{\frac{1}{2}} $, inequalities \eqref{Eq:rew2z} and \eqref{Eq:rew2x} imply inequality \eqref{Eq:rew2}.

For  $z\in \Gamma_{o}^+\cup \Gamma_{x}^+$, we note that
\begin{align*}
   \Re w^2(z)+\Im w^2(z) i= w^2(z)=(\Re w(z))^2-(\Im w(z))^2+2\Re w(z)\Im w(z) i.
\end{align*}
The real part of the left hand side equals  the real part of the right hand side, i.e., 
\begin{align*}
\Re w^2(z)=(\Re w(z))^2-(\Im w(z))^2.   
\end{align*}
Moreover, inequalities \eqref{Eq:rew2z} and  \eqref{Eq:rew2x} yield that 
\begin{align*}
   |\Im w(z)|=\sqrt{(\Re w(z))^2-\Re w^2(z)}>1.   
\end{align*}
Therefore, for all $u \in \mathbb{R}$
\begin{align*}
 | w(z)-u|\geqslant& |\Im (w(z))|>1.
 \end{align*}

\subsection{Proof of Lemma \ref{Lem:imw1}}\label{APP:L7}

  For $z\in [1+\varepsilon,  \cos\gamma+4e]^+=\Gamma_{(o,1)}^+$,  it follows from Tab. \ref{Tab:argfw} that $\arg w\in (\pi, \frac{3\pi}{2})$. Using Remark \ref{rem:log} and equation \eqref{Eq:wz2}, we choose the branch cut
\begin{align}\label{Eq:w21x}
 w^2(z)=\log \frac{2(z-\cos\gamma)}{z^2-1}+i\gamma
=\Big(\log^2\frac{2(z-\cos\gamma)}{z^2-1}+\gamma^2\Big)^{\frac12}e^{i\arg w^2(z)},
\end{align}
where the argument $\arg w^2(z)$ is defined piece-wisely by
\begin{equation}\label{Eq:argwo1}
   \arg w^2(z)= 
    \begin{cases}
        \arctan\frac{\gamma}{\log\frac{2(z-\cos\gamma)}{z^2-1}} +2\pi~~~~~~~&z\in[1+\varepsilon,  1+2\sin\frac{\gamma}{2});\\
        \frac{5\pi}{2} &z=1+2\sin\frac{\gamma}{2};\\
       \arctan\frac{\gamma}{\log\frac{2(z-\cos\gamma)}{z^2-1}}+3\pi&z\in(1+2\sin\frac{\gamma}{2}, \cos\gamma+4e].
    \end{cases}
\end{equation}

Similarly, for $z\in \Gamma^+_{(x,-1)}=[-1+\varepsilon, \cos\gamma-\frac{\sin^2\gamma}{8}]^+$,  it follows from Tab. \ref{Tab:argfw} that $\arg w\in (0,\frac{\pi}{2})$. 
Remark \ref{rem:log} and equation \eqref{Eq:wz2}  imply that  we can represent $w^2(z)$ as in equation \eqref{Eq:w21x},
while the $\arg w^2(z)$ is given by 
\begin{equation}\label{Eq:argw-1x}
   \arg w^2(z)= 
    \begin{cases}
        \arctan\frac{\gamma}{\log \frac{2(z-\cos\gamma)}{z^2-1}}~~~~~~~&z\in[-1+\varepsilon,  1-2\sin\frac\gamma2);\\
        \frac{\pi}{2} &z=1-2\sin\frac\gamma2;\\
       \arctan\frac{\gamma}{\log \frac{2(z-\cos\gamma)}{z^2-1}}+\pi&z\in(1-2\sin\frac\gamma2, x-\frac{\sin^2\gamma}{8}].
    \end{cases}
\end{equation}

 For $z\in \Gamma_{(1,x)}^+=[\cos\gamma+\frac{\sin^2\gamma}{8}, 1-\varepsilon]$,  it follows from Tab. \ref{Tab:argfw} that $\arg w\in (\frac{\pi}{2},\pi)$.  Applying equation \eqref{Eq:wz2} and Remark \ref{rem:log}, we obtain
\begin{align}\label{Eq:w2zx1}
 w^2(z)=\Big(\log^2\frac{2(z-\cos\gamma)}{1-z^2}+(\pi-\gamma)^2\Big)^{\frac12}
e^{i\arg w^2(z)},
\end{align}
where the argument of $w^2(z)$ is represented by
\begin{equation}\label{Eq:argwx1}
   \arg w^2(z)= 
    \begin{cases}
       \arctan\frac{-(\pi-\gamma)}{\log \frac{2(z-\cos\gamma)}{1-z^2}}+\pi ~~~~~~~&z\in[\cos\gamma+\frac{\sin^2\gamma}{8}, -1+2\cos\frac{\gamma}2);\\
        \frac{3\pi}{2} &z=-1+2\cos\frac{\gamma}2;\\
      \arctan\frac{-(\pi-\gamma)}{\log \frac{2(z-\cos\gamma)}{1-z^2}}+2\pi &z\in(-1+2\cos\frac{\gamma}2, 1-\varepsilon].
    \end{cases}
\end{equation}

Similarly, for $z\in \Gamma^+_{(-1,o)}=(-4e+\cos\gamma, -1-\varepsilon)e^{i\pi}$, Tab. \ref{Tab:argfw} gives that $\arg w\in (-\frac{\pi}{2}, 0)$. 
Equation \eqref{Eq:wz2} and Remark \ref{rem:log}  yield that we can represent $w^2(z)$ as in equation \eqref{Eq:w2zx1}, but $ \arg w^2(z) $ is given by 
\begin{equation}\label{Eq:argw-1o}
   \arg w^2(z)= 
    \begin{cases}
        \arctan\frac{(\pi-\gamma)}{\log \frac{z^2-1}{2(\cos\gamma-z)}}-\pi ~~~~~~~&z\in[-4e+\cos\gamma, -1-2\cos\frac{\gamma}{2});\\
        -\frac{\pi}{2} &z=-1-2\cos\frac{\gamma}{2};\\
      \arctan\frac{(\pi-\gamma)}{\log \frac{z^2-1}{2(\cos\gamma-z)}} &z\in(-1-2\cos\frac{\gamma}{2}, -1-\varepsilon].
    \end{cases}
\end{equation}

1.    Given $0<\gamma\leqslant \frac{\pi}{2}$, equations \eqref{Eq:w21x} and \eqref{Eq:w2zx1} imply inequality \eqref{Eq:wz2o-1} for $z\in \Gamma_{(o,1)}^+\cup \Gamma^+_{(x,-1)}\cup \Gamma_{(1,x)}^+\cup \Gamma^+_{(-1,o)}$.

2. Since $z\in\Gamma_{(o,1)}^{+,l}=[1+\varepsilon,  1+\sqrt{1-\cos\gamma}]^+$, we have $\log\frac{z^2-1}{2(z-\cos\gamma)}<0$. From Remark \ref{rem:log},  $\big|\log\frac{z^2-1}{2(z-\cos\gamma)}\big|>0$  decreases in this interval. It follows from
\begin{equation}\label{Eq:log1+r}
    \log(1+r)\geqslant \frac{r}{2}
\end{equation}
for $1>r\geqslant0$, 
we obtain
\begin{align*}
\big|\log\frac{z^2-1}{2(z-\cos\gamma)}\big|\geqslant \big|\log \frac{\sqrt{1-\cos\gamma}(2+\sqrt{1-\cos\gamma})}{2(\sqrt{1-\cos\gamma}+1-\cos\gamma)}\big|= \log \frac{2(1+\sqrt{1-\cos\gamma})}{2+\sqrt{1-\cos\gamma}}\geqslant \frac{\sqrt{1-\cos\gamma}}{6}.
\end{align*}
Recalling that $1-\cos\gamma\geqslant \frac{\gamma^2}{4}$ for $\gamma\in (0,\frac{\pi}{2})$, it follows that
\[
0<\Big|\frac{\gamma}{\log \frac{z^2-1}{2(z-\cos\gamma)}}\Big|\leqslant  12.
\]
For any  $t\in (0,b]$ with $b>0$, we have 
\begin{align}\label{Eq:arctan}
    \arctan t-\frac{t}{1+b^2}\geqslant 0. 
\end{align}
Setting $t=\Big|\frac{\gamma}{\log \frac{z^2-1}{2(z-\cos\gamma)}}\Big|$  and $b=12$ , we obtain 
\begin{align}\label{Eq:artano1}
  \arctan\Big|\frac{\gamma}{\log \frac{z^2-1}{2(z-\cos\gamma)}}\Big|\geqslant \Big|\frac{\gamma}{145\log \frac{z^2-1}{2(z-\cos\gamma)}}\Big|.   
\end{align}
In addition, for $ ~0<\tau<1 ~{ \rm and }~ 0<t<1$, we can easily check that
\begin{align}\label{Eq:logt}
   \frac{1}{\sqrt{ |\log t|}}> \sqrt{\tau}\,t^{\frac\tau 2}.
\end{align}
Since $|\frac12\arctan\frac{\gamma}{\log \frac{z^2-1}{2(\cos\gamma-z)}}|\leqslant \frac{\pi}4$ and $\Gamma_{(o,1)}^{+,l}\subset [-1+\varepsilon,  1-2\sin\frac\gamma2)$, it follows from equation \eqref{Eq:w21x} and the first equality in equation \eqref{Eq:argwo1}  that  
\begin{align*}
|\Im w(z)|=&\Big|\Big(\log^2 \frac{z^2-1}{2(z-\cos\gamma)}+\gamma^2\Big)^{\frac14}\sin\big(\frac12\arctan\frac{\gamma}{\log\frac{z^2-1}{2(z-\cos\gamma)}}\big)\Big|.
\end{align*}
From inequalities \eqref{Eq:singa} and \eqref{Eq:artano1}, respectively, it follows that
\begin{align*}
|\Im w(z)| 
\geqslant |\log\frac{z^2-1}{2(z-\cos\gamma)}|^{\frac12}\frac14 \arctan\Big|\frac{\gamma}{\log \frac{z^2-1}{2(z-\cos\gamma)}}\Big|  
\geqslant  |\log\frac{z^2-1}{2(z-\cos\gamma)}|^{\frac12}\big(\frac{\gamma}{580\log\frac{z^2-1}{2(z-\cos\gamma)}}\big).
\end{align*}
Given that $0< \frac{z^2-1}{2(z-\cos\gamma)}<1$, by inequality \eqref{Eq:logt} with $t=\frac{z^2-1}{2(z-\cos\gamma)}$ and $\tau=\alpha+1$, inequality \eqref{Eq:imwz1o1} holds for $\Gamma_{(o,1)}^{+,l}$.

 As $z\in \Gamma_{(1,x)}^{+,r}=[\cos\frac{\gamma}{2}, 1-\varepsilon ]$,
Remark \ref{rem:log} implies that  $\log \frac{2(z-\cos\gamma)}{1-z^2} $ increases in this region.  The 
 following equalizes 
\begin{align}\label{Eq:sincos2an}
  \frac{2(\cos\frac{\gamma}{2}-\cos\gamma)}{\sin^2\frac{\gamma}{2}}=\frac{2(\cos\frac{\gamma}{2}-\cos^2\frac{\gamma}{2}+\sin^2\frac{\gamma}{2})}{\sin^2\frac{\gamma}{2}}=\frac{4\cos\frac{\gamma}2\sin^2\frac{\gamma}{4}}{4\sin^2\frac{\gamma}{4}\cos^2\frac{\gamma}{4}}+2=2+\frac{\cos\frac{\gamma}2}{\cos^2\frac{\gamma}{4}}.
\end{align}
yields that 
\begin{align*}
\log \frac{2(z-\cos\gamma)}{1-z^2}\geqslant \log \frac{2(\cos\frac{\gamma}{2}-\cos\gamma)}{\sin^2\frac{\gamma}{2}}= \log (2+\frac{\cos\frac{\gamma}2}{\cos^2\frac{\gamma}{4}})>\log 2>\frac{1}{2}.
\end{align*}
This implies that 
\[
\Big|\frac{\pi-\gamma}{\log \frac{2(z-\cos\gamma)}{1-z^2}}\Big|\leqslant  2(\pi-\gamma).
\]
 In this case,  from inequality \eqref{Eq:arctan} with $t=|\frac{\pi-\gamma}{\log \frac{2(z-\cos\gamma)}{1-z^2}}\Big| $ and  $b=2(\pi-\gamma)$, we have 
\begin{align}\label{Eq:artanx1}
  \arctan\Big|\frac{\pi-\gamma}{\log \frac{2(z-\cos\gamma)}{1-z^2}}\Big|\geqslant \Big|\frac{\pi-\gamma}{(1+4(\pi-\gamma)^2)\log \frac{2(z-\cos\gamma)}{1-z^2}}\Big|\geqslant \Big|\frac{1}{5(\pi-\gamma)\log \frac{2(z-\cos\gamma)}{1-z^2}}\Big|.   
\end{align}
Since $|\frac12\arctan\frac{(\pi-\gamma)}{\log \frac{2(z-\cos\gamma)}{1-z^2}}|\leqslant \frac{\pi}4$ and $ \Gamma_{(1,x)}^{+,r}\subset (-1-2\cos\frac{\gamma}{2}, -1-\varepsilon]$,  it follows from equations \eqref{Eq:w2zx1} and the third formula of equation \eqref{Eq:argwx1}  that 
\begin{align*}
|\Im w(z)|=&\Big|\Big(\log^2\frac{2(z-\cos\gamma)}{1-z^2}+(\pi-\gamma)^2\Big)^{\frac14}\sin\big(\frac12\arctan\frac{\pi-\gamma}{\log \frac{2(z-\cos\gamma)}{1-z^2}}\big)\Big|.
\end{align*}
The inequalities \eqref{Eq:singa} and \eqref{Eq:artanx1} yield respectively, that
\begin{align*}
|\Im w(z)|
\geqslant\left|\log \frac{2(z-\cos\gamma)}{1-z^2}\right|^{\frac12}\frac14\arctan\frac{\pi-\gamma}{\log \frac{2(z-\cos\gamma)}{1-z^2}}
\geqslant  \left|\log\frac{2(z-\cos\gamma)}{1-z^2}\right|^{\frac12}\big(\frac{1}{20(\pi-\gamma)\log \frac{2(z-\cos\gamma)}{1-z^2}}\big).
\end{align*}
Given that $0<\frac{1-z^2}{2(z-\cos\gamma)}<1$, applying inequality \eqref{Eq:logt} with $t=\frac{z^2-1}{2(\cos\gamma-z)}$ and $\tau=\alpha+1$, we have 
\begin{equation*}
|\Im w(z)|
\geqslant \frac{\sqrt{\alpha+1}}{20(\pi-\gamma)} \Big|\frac{z+1}{2(z-\cos\gamma)}\Big|^{\frac{\alpha+1}2} |z-1|^{\frac{\alpha+1}2}.
\end{equation*}
Recalling that $0<\gamma\leqslant \frac{\pi}{2}$, we have 
\[
 \frac{\gamma\sqrt{\alpha+1}}{580} < \frac{\sqrt{\alpha+1}}{20(\pi-\gamma)}. 
\]
Thus, inequality \eqref{Eq:imwz1o1} holds for $\Gamma_{(1,x)}^{+,r}$.
 
3. As $z\in \Gamma_{(-1,o)}^{+,r}=(-1-\sqrt{\cos\gamma+1}, -1-\varepsilon]$, we have $\log\frac{z^2-1}{2(\cos\gamma-z)}<0$. Remark \ref{rem:log} implies that $\big|\log\frac{z^2-1}{2(\cos\gamma-z)}\big|>0$ increases in this region. By inequality \eqref{Eq:log1+r}, a direct calculation shows that
\begin{align*}
\big|\log \frac{z^2-1}{2(\cos\gamma-z)}\big|\geqslant \big|\log \frac{(2+\sqrt{\cos\gamma+1})\sqrt{\cos\gamma+1}}{2(1+\sqrt{\cos\gamma+1}+\cos\gamma)}\big|= \log \frac{2(1+\sqrt{\cos\gamma+1})}{2+\sqrt{\cos\gamma+1}}\geqslant\frac{1}{6}.
\end{align*}
This yields that
\[
\Big|\frac{(\pi-\gamma)}{\log \frac{z^2-1}{2(\cos\gamma-z)}}\Big|\leqslant  6(\pi-\gamma).
\]
 In this case, recalling  $0<\gamma<\frac{\pi}2$ and applying inequality \eqref{Eq:arctan} with $t=\Big|\frac{(\pi-\gamma)}{\log \frac{z^2-1}{2(\cos\gamma-z)}}\Big|$ and $b= 6(\pi-\gamma)$, we have 
\begin{align}\label{Eq:artan-1o}
  \arctan\Big|\frac{(\pi-\gamma)}{\log \frac{z^2-1}{2(\cos\gamma-z)}}\Big|\geqslant \Big|\frac{(\pi-\gamma)}{(1+36(\pi-\gamma)^2)\log \frac{z^2-1}{2(\cos\gamma-z)}}\Big|\geqslant \Big|\frac{1}{37(\pi-\gamma)\log \frac{z^2-1}{2(\cos\gamma-z)}}\Big|.   
\end{align}
Since $|\frac12\arctan\frac{(\pi-\gamma)}{\log \frac{z^2-1}{2(\cos\gamma-z)}}|\leqslant \frac{\pi}4$ and $\Gamma_{(-1,o)}^{+,r}\subset (-1-2\cos\frac{\gamma}{2}, -1-\varepsilon] $ , it follows from  equation \eqref{Eq:w2zx1} and the third formula of equation \eqref{Eq:argw-1o} that 
\begin{align*}
|\Im w(z)|=&\Big|\Big(\log^2\frac{2(z-\cos\gamma)}{1-z^2}+(\pi-\gamma)^2\Big)^{\frac14}\sin\big(\frac12\arctan\frac{\pi-\gamma}{\log \frac{2(z-\cos\gamma)}{1-z^2}}\big)\Big|.
\end{align*}
The inequalities \eqref{Eq:singa} and \eqref{Eq:artan-1o} give
\begin{align*}
|\Im w(z)|
\geqslant& |\log \frac{z^2-1}{2(\cos\gamma-z)}|^{\frac12}\frac14\arctan\Big|\frac{(\pi-\gamma)}{\log \frac{z^2-1}{2(\cos\gamma-z)}}\Big| 
\geqslant  |\log \frac{z^2-1}{2(\cos\gamma-z)}|^{\frac12}\Big|\frac{1}{148(\pi-\gamma)\log \frac{z^2-1}{2(\cos\gamma-z)}}\Big|.
\end{align*}
Since $0< \frac{z^2-1}{2(\cos\gamma-z)}<1$, applying inequality \eqref{Eq:logt} with $t=\frac{z^2-1}{2(\cos\gamma-z)}$ and $\tau=\beta+1$, we have 
\begin{equation}\label{Eq:imwz-1or} 
|\Im w(z)|\geqslant \frac{\sqrt{\beta+1}}{148(\pi-\gamma)} \Big|\frac{z-1}{2(z-\cos\gamma)}\Big|^{\frac{\beta+1}2} |z+1|^{\frac{\beta+1}2}
\end{equation}

 As $z\in \Gamma_{(x,-1)}^{+,l}=[-1+\varepsilon, 1-\sqrt 3]$, Remark \ref{rem:log} implies that $\log\frac{2(z-\cos\gamma)}{z^2-1}>0$ decreases in this region, and we have
\begin{align*}
\log \frac{2(z-\cos\gamma)}{z^2-1}\geqslant  \log \frac{2(\cos\gamma-1+\sqrt 3)}{1-(1-\sqrt3)^2}\geqslant \log \frac{(2\sqrt{3}-2)(2\sqrt{3}+3)}{3}>1.
\end{align*}
This implies that 
\[
\Big|\frac{\gamma}{\log \frac{2(z-\cos\gamma)}{z^2-1}}\Big|\leqslant  \gamma.
\]
 In this case, using  inequality \eqref{Eq:arctan} with $t=\Big|\frac{\gamma}{\log \frac{2(z-\cos\gamma)}{z^2-1}}\Big|$ and $b=\gamma$, we have 
\begin{align}\label{Eq:artan-1x}
  \arctan\Big|\frac{\gamma}{\log \frac{2(z-\cos\gamma)}{z^2-1}}\Big|\geqslant \Big|\frac{\gamma}{(1+\gamma^2)\log \frac{2(z-\cos\gamma)}{z^2-1}}\Big|.   
\end{align}
Since $|\frac12\arctan\frac{\gamma}{\log \frac{z^2-1}{2(z-\cos\gamma)}}|\leqslant \frac{\pi}4$ and $\Gamma_{(x,-1)}^{+,l}\subset [-1+\varepsilon,  1-2\sin\frac\gamma2)$, it follows from equation \eqref{Eq:w21x} and the first formula in equation \eqref{Eq:argw-1x} that 
\begin{align*}
|\Im w(z)|
=&\Big|\Big(\log^2\frac{2(z-\cos\gamma)}{z^2-1}+\gamma^2\Big)^{\frac14}\sin\big(\frac12\arctan\frac{\gamma}{\log \frac{2(z-\cos\gamma)}{z^2-1}}\big)\Big|.
\end{align*}
Furthermore, it follows from inequalities \eqref{Eq:singa} and \eqref{Eq:artan-1x} respectively, that
\begin{align*}
|\Im w(z)|
\geqslant& |\log \frac{2(z-\cos\gamma)}{z^2-1}|^{\frac12}\frac14\Big|\arctan\frac{\gamma}{\log \frac{2(z-\cos\gamma)}{z^2-1}}\Big|
\geqslant  |\log \frac {z^2-1}{2(z-\cos\gamma)}|^{\frac12}\big(\frac{\gamma}{4(1+\gamma^2)|\log \frac {z^2-1}{2(z-\cos\gamma)}|}\big).
\end{align*}
Given that $0<\frac{z^2-1}{2(z-\cos\gamma)}<1$, by  inequality \eqref{Eq:logt} with $t=\frac{z^2-1}{2(z-\cos\gamma)}$  and $\tau=\beta+1$, we obtain
\begin{equation}\label{Eq:imwzx-1l} 
|\Im w(z)|
\geqslant \frac{\gamma\sqrt{\beta+1}}{4(1+\gamma^2)} \Big|\frac{z-1}{2(z-\cos\gamma)}\Big|^{\frac{\beta+1}2} |z+1|^{\frac{\beta+1}2}.
\end{equation}
Given $0<\gamma\leqslant \frac{\pi}{2}$, we have 
\[
\min\left\{ \frac{\sqrt{\beta+1}}{148 (\pi-\gamma)}, \frac{\gamma\sqrt{\beta+1}}{4(1+\gamma^2)} 
\right \}\geqslant \frac{\gamma\sqrt{\beta+1}}{74 \pi(\pi-\gamma)}.
\]
Therefore, for $z\in \Gamma_{(-1,o)}^{+,r}\cup \Gamma_{(x,-1)}^{+,l}$,  the inequalities \eqref{Eq:imwz-1or} and \eqref{Eq:imwzx-1l} imply \eqref{Eq:imwz2-1}.

\subsection{Proof of Lemma \ref{Le:imw}}\label{APP:L8}

Set $w(z)=\mu+i\nu$.
For $z\in \Gamma_{(o,1)}^+$ or  $z\in \Gamma_{(x,-1)}^+$,  it follows from equation \eqref{Eq:wz2} that 
\begin{equation*}\label{Eq:wML}
\begin{split}
\mu+i\nu
&=\Big[-\log\frac{z^2-1}{2(z-\cos\gamma)}+\gamma i\Big]^{1/2} .
\end{split}
\end{equation*}
Since the imaginary parts of the two sides are equal, we have 
\[
\mu =\frac{\gamma}{2\nu} \qquad \text{ and } \qquad \mu^2-\nu^2=-\log\frac{z^2-1}{2(z-\cos\gamma)}.
\]
By a simple calculation, we obtain 
\begin{align}\label{Eq:nu1o}
\nu^2=\frac{\log\frac{z^2-1}{2(z-\cos\gamma)}+\sqrt{\log^2\frac{z^2-1}{2(z-\cos\gamma)}+\gamma^2}}2.
\end{align}

For $z \in\Gamma_{(o,1)}^{+,r} = (1+\sqrt{1-\cos\gamma}, \cos\gamma+4e]^+$ or $z \in\Gamma_{(x,-1)}^{+,r}=[1-\sqrt 3, \cos\gamma-\frac{\sin^2\gamma}{8}]^+$, by Remark \ref{rem:log}, $\log \frac{z^2-1}{2(z-\cos\gamma)}$ is increasing, and
 we have
\begin{equation}\label{Eq:log1o}
\log\frac{z^2-1}{2(z-\cos\gamma)}\geqslant\min\{\log\frac{2+\sqrt{1-\cos\gamma}}{2(1+\sqrt{1-\cos\gamma})}, \log \frac{1-(1-\sqrt3)^2}{2(\cos\gamma-1+\sqrt 3)}\}> \log \frac{2\sqrt{3}-3}{2\sqrt{3}}>-3.
\end{equation}
It follows from equation \eqref{Eq:nu1o} that $|\Im w(z)|^2=\nu^2$ is increasing with respect to $\log\frac{z^2-1}{2(z-\cos\gamma)}$.
Therefore, inequality \eqref{Eq:log1o} 
implies that 
\[
\nu^2\geqslant \frac{-3+\sqrt{3^2+\gamma^2}}2.
\]
For $0<\gamma\leqslant \frac{\pi}2$, using  the inequality 
\[\gamma^2\geqslant\Big(\frac{\gamma^2}{6\pi }\Big)^2+\frac{\gamma^2}{\pi} ,\]
 inequality \eqref{Eq:imwz-1o2} holds for $z \in \Gamma_{(o,1)}^{+,r}\cup\Gamma_{(x,-1)}^{+,r}$.

 Similarly,  when $z\in\Gamma_{(1,x)}^{+}\cup \Gamma_{(-1,o)}^{+}$, it follows from equation \eqref{Eq:wz2} that 
\[
\mu =\frac{\gamma-\pi}{2\nu}
\]
and 
\[
\mu^2-\nu^2=-\log\frac{z^2-1}{2(\cos\gamma-z)}.
\]
Using the last two equations, one can show
\begin{align}\label{Eq:nuo-1}
\nu^2=\frac{\log\frac{z^2-1}{2(\cos\gamma-z)}+\sqrt{\log^2\frac{z^2-1}{2(\cos\gamma-z)}+(\pi-\gamma)^2}}2.
\end{align}

For $z\in\Gamma_{(-1,o)}^{+,l}= [-4e+\cos\gamma, -1-\sqrt{\cos\gamma+1}]^+ $ or $z \in\Gamma_{(1,x)}^{+,l}=[\cos\gamma+\frac{\sin^2\gamma}{8},\cos\frac{\gamma}{2} ]$,  
we recall from Remark \ref{rem:log} that $\log \frac{z^2-1}{2(\cos\gamma-z)}$ is decreasing in these regions respectively. Thus,
\begin{equation}\label{Eq:logo1}
\log \frac{z^2-1}{2(\cos\gamma-z)}\geqslant\min\{ -\log \frac{2(1+\sqrt{\cos\gamma+1})}{2+\sqrt{\cos\gamma+1}}, -\log\Big(2+\frac{\cos\frac{\gamma}2}{\cos^2\frac{\gamma}{4}}\Big)\} \geqslant-2.
\end{equation}
Here, we have used equation \eqref{Eq:sincos2an} to calculate $-\log  \frac{2(\cos\frac{\gamma}{2}-\cos\gamma)}{\sin^2\frac{\gamma}{2}}$ . 
From equation \eqref{Eq:nuo-1}, it follows that $|\Im w(z)|^2=\nu^2$ is increasing with respect to $\log\frac{z^2-1}{2(\cos\gamma-z)}$. 
Therefore,  inequality \eqref{Eq:logo1} yields that 
\[
\nu^2\geqslant \frac{-2+\sqrt{(-2)^2+(\pi-\gamma)^2}}2.
\]
For $0<\gamma\leqslant \frac{\pi}2$, using inequality 
\[
(\pi-\gamma)^2\geqslant\Big(\frac{(\pi-\gamma)^2}{2\pi }\Big)^2+\frac{2(\pi-\gamma)^2}{\pi},
\]
we have
\begin{align*}
 |v|\geqslant\frac{\pi-\gamma}{2\sqrt{\pi}}
>\frac{\gamma}{2\sqrt{ 3\pi}}.
\end{align*}
This implies the inequality \eqref{Eq:imwz-1o2} for $z\in \Gamma_{(-1,o)}^{+,l}\cup \Gamma_{(1,x)}^{+,l}$.

\subsection{Proof of Lemma \ref{Le:g2}}\label{APP:L9}

1. As $z\in  \Gamma_{-1}^+$, we have $2-\varepsilon \leqslant |z-1|\leqslant 2+\varepsilon$ and $|z-\cos\gamma|\geqslant1+\cos\gamma -\varepsilon$.
Furthermore, we have
\[
\lim_{\varepsilon\to 0}\max\{(2-\varepsilon)^{\alpha}, (2+\varepsilon)^{\alpha}\}<1.\] 
Thus, inequalities in \eqref{Eq:gzGa-1-x} hold immediately.

2. Similarly, we have the inequalities in \eqref{Eq:gzGa1-x}. 

3.
 We only need to check the last two formulas in \eqref{Eq:Gaoz-x1-1}. Since $4e+\cos\gamma-1\leqslant |z-1|\leqslant 4e-\cos\gamma+1$  and $ 4e-\cos\gamma-1\leqslant |z+1|\leqslant 4e+\cos\gamma+1$ for  $z\in \Gamma_{o}^+$, recalling  $-1<\alpha\leqslant 0$, and $\beta\leqslant \alpha$, we have 
 \begin{align*}
\max \Big\{\Big(4e+\cos\gamma-1\Big)^{\alpha}, \Big(4e-\cos\gamma+1\Big)^{\alpha}, \Big(4e-\cos\gamma-1\Big)^{\beta},\Big(4e+\cos\gamma+1\Big)^{\beta}\Big\}<1.\end{align*}
Therefore,  \eqref{Eq:Gaoz-x1-1} holds. 

4. For $z\in \Gamma_x^+$, we also only need to check the last two formulas in \eqref{Eq:Gaxz-x1-1}. 
 Since $1-\cos\gamma-\frac{\sin^2\gamma}{8}\leqslant |z-1|\leqslant 1-\cos\gamma+\frac{\sin^2\gamma}{8}$  and $ 1+\cos\gamma-\frac{\sin^2\gamma}{8}\leqslant |z+1|\leqslant 1+\cos\gamma+\frac{\sin^2\gamma}{8}$ in this half circle, using the inequality \eqref{Eq:singa},
we have
\[
1-\cos\gamma-\frac{\sin^2\gamma}{8}= \frac{\sin^2\gamma}{2\cos^2\frac{\gamma}{2}}-\frac{\sin^2\gamma}{8}\geqslant \frac{\sin^2\gamma}3\geqslant\frac{\gamma^2}{12}.
\]
Recalling that $-1<\alpha\leqslant 0$,  $\beta\leqslant \alpha$, we have 
 \begin{align*}
\max \Big\{\Big(1-\cos\gamma-\frac{\sin^2\gamma}{8}\Big)^{\alpha}, \Big(1-\cos\gamma+\frac{\sin^2\gamma}{8}\Big)^{\alpha}\Big\}\leqslant 12 \gamma^{2\alpha}
 \end{align*}
 and 
 \begin{align*}
 \max \Big\{\Big(1+\cos\gamma-\frac{\sin^2\gamma}{8}\Big)^{\beta},\Big(1+\cos\gamma+\frac{\sin^2\gamma}{8}\Big)^{\beta}\Big\}\leqslant2.
 \end{align*}
Therefore, \eqref{Eq:Gaxz-x1-1} holds.

\end{appendices}


\begin{thebibliography}{99}

\bibitem{BHNOD18}
 T. Bennett, C. J. Howls, G. Nemes and A. B. Olde Daalhuis. Globally exact asymptotics
for integrals with arbitrary order saddles, {\it SIAM J. Math. Anal.} {\bf 50}, (2018), 2144–2177.

\bibitem{BH91}
 M. V. Berry and C. J. Howls. Hyperasymptotics for integrals with saddles, {\it Proc. Roy. Soc.
London Ser. A}  {\bf 434}, (1991), 657–675.

\bibitem{B93}
 W. G. C. Boyd. Error bounds for the method of steepest descents, {\it Proc. Roy. Soc. London
Ser. A}  {\bf 440}, (1993), 493–518.

\bibitem{CI91}
L.-C. Chen and M.~E.-H. Ismail, On asymptotics of Jacobi polynomials, {\it SIAM J. Math. Anal.} {\bf 22} (1991), 1442-1449.


\bibitem{DKMVZ99}
 P. Deift, T. Kriecherbauer, K. T.-R. McLaughlin, S. Venakides and X. Zhou. Strong asymptotics of orthogonal polynomials with respect to exponential weights, {\it Comm. Pure Appl.
Math.}  {\bf 52}, (1999), 1491–1552.

\bibitem{DZ93}
P. Deift and X. Zhou. A steepest descent method for oscillatory Riemann–Hilbert problems.
Asymptotics for the MKdV equation, {\it Ann. Math.} {\bf 137}, (1993), 295–368.

\bibitem{KV99}
A. B. J. Kuijlaars and W. Van Assche,
    The asymptotic zero distribution of orthogonal polynomials
              with varying recurrence coefficients,
    {\it J. Approx. Theory},
   {\bf 99},
  (1999), 167-197.


\bibitem{KMVV04}
A. B. J. Kuijlaars, K. T.-R. McLaughlin, W. Van Assche,
            and M. Vanlessen,
  The {R}iemann-{H}ilbert approach to strong asymptotics for orthogonal polynomials on {$[-1,1]$}, {\it Adv. Math.}, {\bf 188}, (2004), {337-398}.

     

\bibitem{L51}
 H. A. Lauwerier. The calculation of the coefficients of certain asymptotic series by means of linear recurrent relations, {\it Appl. Sci. Res. B} {\bf 2}, (1951), 77–84.



\bibitem{Olver74}
 F. W. J. Olver. Asymptotics and Special Functions (Academic Press, New York, 1974).

\bibitem{Olver80}
 F. W. J. Olver. Asymptotic approximations and error bounds. SIAM Rev. 22 (1980),
188–203.

\bibitem{SNWW23}
W. Shi, G. Nemes, X.-S. Wang and R. Wong. Error bounds for the asymptotic expansions of the Hermite polynomials, {\it Proc. Roy. Soc. Edinburgh Sect. A}  {\bf 153(2)}, (2023), 417-440.  

\bibitem{Szego79}
G. Szeg$\ddot{\rm o}$. Orthogonal Polynomials, Fourth ed., in: American Mathematical Society Colloquium Publications, vol.
23, American Mathematical Society, Providence, RI, 1975.


\bibitem{Wong80}
R. Wong. Error bounds for asymptotic expansions of integrals. SIAM Rev. 22 (1980),
401–435.

\bibitem{Wong89}
 R. Wong. Asymptotic Approximations of Integrals (Academic Press, Boston, 1989).

 \bibitem{Wong14}
 R. Wong. Asymptotics of linear recurrences. Anal. Appl. 12 (2014), 463–484.



\bibitem{WZ96}
R. Wong and J.-M. Zhang, A uniform asymptotic expansion for the Jacobi polynomials with explicit remainder, {\it Appl. Anal.} {\bf 61}, (1996), 17–29.


 \bibitem{WZ05}
 R. Wong and Y.-Q. Zhao. On a uniform treatment of Darboux’s method. Constr. Approx.
21 (2005), 225–255.

\bibitem{WZ03}
 R. Wong and Y.-Q. Zhao. Estimates for the error term in a uniform asymptotic expansion of
the Jacobi polynomials, {\it Anal. Appl.} {\bf 1}, (2003), 213–241.
\end{thebibliography}
\end{document}